\documentclass[english,11pt]{article}
\usepackage[margin=1in]{geometry}
\usepackage{color,cite}
\usepackage{verbatim}
\usepackage{float}
\usepackage{hyperref,url}

\usepackage{mathtools}
\usepackage{todonotes}
\usepackage{amsmath}
\usepackage{amsthm,enumerate}
\usepackage[capitalise,nameinlink,noabbrev]{cleveref}
\usepackage{amssymb}
\usepackage{xargs}[2008/03/08]
\newcommand{\dist}{\textsf{dist}}
\usepackage[figurewithin=section]{caption}

\usepackage[noend]{algpseudocode}
\floatstyle{ruled}
\newfloat{algorithm}{tbp}{loa}
\floatname{algorithm}{Procedure}

\theoremstyle{plain}
\newtheorem{thm}{Theorem}[section]
\newtheorem{cor}[thm]{Corollary}
\crefname{cor}{corollary}{corollaries}
\crefformat{cor}{#2Corollary~#1#3}
\Crefformat{cor}{#2Corollary~#1#3}
\newtheorem{claim}{Claim}

\newtheorem{observation}[thm]{Observation}
\newtheorem{lem}[thm]{Lemma}
\newtheorem{conjecture}[thm]{Conjecture}
\newtheorem{prop}[thm]{Proposition}
\crefname{prop}{Proposition}{Propositions}
\crefformat{prop}{#2Proposition~#1#3}
\Crefformat{prop}{#2Proposition~#1#3}

\newtheorem{rproblem}[thm]{Research Problem}

\newtheorem{question}[thm]{Question}

\theoremstyle{definition}
\newtheorem{defn}[thm]{Definition}
\crefname{defn}{definition}{definitions}
\crefformat{defn}{#2Definition~#1#3}
\Crefformat{defn}{#2Definition~#1#3}
\newtheorem{example}[thm]{Example}
\theoremstyle{remark}
\newtheorem{rem}[thm]{Remark}
\theoremstyle{plain}

\newcommand{\Hst}{(\hat H, \hat s, \hat t)}

\newcommandx\s[1][usedefault, addprefix=\global, 1=]{\underset{#1}{\rightarrow}}

\def\mpfile#1#2{\includegraphics{figures/#1-#2-mps.pdf}}
\def\orb#1{\mathop{\mathrm{Orb}}(#1)}

\def\B{\mathcal B}
\let\es\varnothing
\def\rank#1{\mathop{\mathrm{rk}}(#1)}

\usepackage{babel}
\newcommand{\ourtitle}{Avoidability beyond paths}

\colorlet{darkred}{red!80!black}
\colorlet{darkblue}{blue!80!black}
\colorlet{darkgreen}{green!60!black}
\colorlet{myGreen}{green!50!black}
\definecolor{myBlue}{rgb}{0.25, 0.0, 1.0}
\definecolor{lgray}{rgb}{0.75, 0.75, 0.75}
\hypersetup{
  colorlinks=true,
  linkcolor=darkblue,
  citecolor=darkgreen,
  urlcolor=darkblue,
  bookmarksopen=true,
  bookmarksnumbered,
  bookmarksopenlevel=2,
  bookmarksdepth=3
  pdftitle = {\ourtitle},
  pdfauthor= {Vladimir Gurvich, Matjaž Krnc, Martin Milanič, Mikhail Vyalyi},
}

\title{\ourtitle}
\author{
Vladimir Gurvich\thanks{National Research University Higher School of Economics and RUTCOR, Rutgers University, New Jersey, USA, \texttt{vladimir.gurvich@gmail.com}.}\and
Matjaž Krnc\thanks{FAMNIT and IAM, University of Primorska, \texttt{matjaz.krnc@upr.si}.}\and
Martin Milanič\thanks{FAMNIT and IAM, University of Primorska, \texttt{martin.milanic@upr.si}.}\and
Mikhail Vyalyi\thanks{National Research University Higher School of Economics; Moscow Institute of Physics and Technology; Federal Research Center “Computer Science and Control” of the Russian Academy of Science, \texttt{vyalyi@gmail.com}.}}
\begin{document}
\date{}
\maketitle

\begin{abstract}
The concept of avoidable paths in graphs was introduced by Beisegel, Chudnovsky, Gurvich, Milani\v{c}, and Servatius in 2019 as a common generalization of avoidable vertices and simplicial paths. 
In 2020, Bonamy, Defrain, Hatzel, and Thiebaut proved that every graph containing an induced path of order $k$ also contains an avoidable induced path of the same order. 
They also asked whether one could generalize this result to other avoidable structures, leaving the notion of avoidability up to interpretation.
In this paper we address this question: we specify the concept of avoidability for arbitrary graphs equipped with two terminal vertices.
We provide both positive and negative results, some of which are related to a recent work by Chudnovsky, Norin, Seymour, and Turcotte in 2024.
We also discuss several open questions.

\bigskip


\bigskip

\end{abstract}

\newpage
\tableofcontents
\newpage

\section{Introduction}

\subsection{Motivation}

A graph is \emph{chordal} if it contains no induced cycles of length at least four. 
A classical result of Dirac from 1961~\cite{MR130190} states that every (non-null) chordal graph has a simplicial vertex, that is, a vertex whose neighborhood is a clique. 
This result was generalized in the literature in various ways. 

First, in 1976, Ohtsuki, Cheung, and Fujisawa~\cite{MR485552} generalized Dirac's result from the class of chordal graphs to the class of all graphs using the concept of avoidable vertices.
A vertex $v$ in a graph $G$ is said to be \emph{avoidable} if every induced $3$-vertex path with midpoint $v$ is contained in an induced cycle.
Ohtsuki et al.\ proved that
an avoidable vertex is inherently present in every (non-null) graph.
In fact, it was later discovered that several well-known graph searches such as LexBFS or LexDFS always end in an avoidable vertex (see~\cite{MR2659379,MR1626534,MR408312}).\footnote{The term ``avoidable vertex'' was introduced by Beisegel, Chudnovsky, Gurvich, Milani\v{c}, and Servatius~\cite{BCGMS19,BCGMS19DAM}. 
Avoidable vertices were also called \emph{OCF-vertices} in the literature (see~\cite{MR2659379,MR2057267}).} 

Second, in 2002, Chv\'{a}tal, Rusu, and Shritaran~\cite{MR1927566} generalized Dirac's result from the class of chordal graphs to classes of graphs excluding all sufficiently long induced cycles, by generalizing the concept of simpliciality from vertices to longer induced paths.
An \emph{extension} of an induced path $P$ in a graph $G$ is any induced path in $G$ that can be obtained by extending $P$ by one edge from each endpoint.
An induced path in a graph $G$ is said to be \emph{simplicial} if it has no extensions.
Chv\'{a}tal et al.\ proved that for every positive integer $k$, every graph without induced cycles of length at least $k+3$ either contains no induced $k$-vertex path, or contains a simplicial induced $k$-vertex path.

In 2019, Beisegel et al.~\cite{BCGMS19DAM} proposed a common generalization of all these results, by introducing the concept of \emph{avoidable paths}, a common generalization of avoidable vertices and simplicial paths. 
An induced path $P$ in a graph $G$ is said to be \emph{avoidable} if every extension of $P$ is contained in an induced cycle.
Beisegel et al.\ conjectured that for every positive integer $k$, 
an avoidable $k$-vertex path is inherently present in every graph that contains an induced $k$-vertex path, and proved the statement for the case $k=2$.
The general conjecture was proved in 2020 by Bonamy, Defrain, Hatzel, and Thiebaut~\cite{MR4245221}.
A further strengthening was given by Gurvich et al.~\cite{GKMV20+}, who showed that in every graph, every induced path can be transformed into an avoidable one via a sequence of shifts (where two induced $k$-vertex paths are said to be \emph{shifts} of each other if their union is an induced path with $k+1$ vertices).
In~\cite{GKMV20+}, analogous questions were also considered for general (not necessarily induced) paths, isometric paths, trails, and walks.

Bonamy et al.\ concluded their paper~\cite{MR4245221} with a discussion on whether one can obtain other avoidable structures.
They pointed out that in some cases (e.g., for cliques) the very notion of extension becomes unclear and formulated the following question.
\begin{quote}
    \emph{Does there exist a family $\mathcal{H}$ of connected graphs, not containing any path, such that any graph is either $\mathcal{H}$-free or contains an avoidable element of $\mathcal{H}$?}
\end{quote}
They left the notion of avoidability in this context up to interpretation.

In this paper we address the above question and suggest a framework for studying avoidability in the context of arbitrary graphs and not only paths. 

\subsection{Our approach}

In order to generalize the concept of an extension of a path to that of an arbitrary graph $H$, the role of the endpoints of a path is taken by an arbitrary (but fixed) pair of vertices $s$ and $t$ in $H$ called \emph{roots}.
This naturally leads to a suitable definition of avoidability of a two-rooted graph $(H,s,t)$ (see \cref{def:avoidable}).
Accordingly, we say that a two-rooted graph $(H,s,t)$ is \emph{inherent} if every graph that contains a copy of $H$ also has an avoidable copy of $(H,s,t)$ (see \cref{inherent}).
In this terminology, a result of Bonamy et al.~\cite{MR4245221} states:

\begin{thm}\label{th:endpoint-rooted paths are inherent}
All paths are inherent with respect to their endpoints.
\end{thm}

We provide several necessary conditions for inherence.
We do this by developing a technique for proving non-inherence of two-rooted graphs, which we call the pendant extension method.
In particular, we show that the inherence of paths depends on the choice of the two roots $s$ and $t$.

On the positive side, we develop a technique for proving inherence.
We apply this method to prove \cref{th:endpoint-rooted paths are inherent}. 
We indicate the following interesting open problem.

\begin{conjecture}\label{conj-stv}
Let $P=(s,t,v)$ be a two-edge path. 
Then, the two-rooted path $(P,s,t)$ is inherent.
\end{conjecture}

\noindent
The conjecture can be expressed in a self-contained way as follows.

\begin{sloppypar}
\addtocounter{thm}{-1}
\begin{conjecture}[reformulated]\label{conj:stv}
Every graph $G$ is either a disjoint union of complete graphs, or it contains an induced $P_3=(s,t,v)$ with the following property: 
For any selection of $x\in N(s)$ and $y\in N(t)$ such that $\{s,t,v,x,y\}$ induces a fork\footnote{For the definition of the fork graph, see \Cref{pic:F211}.} in $G$, vertices $x,y$ lie in the same component of \hbox{$G-(N[\{s,t,v\}]\setminus\{x,y\})$}.
\end{conjecture}
\end{sloppypar}

Despite the small size of the two-rooted graph, resolving this conjecture seems to be difficult.
Recently, Chudnovsky, Norin, Seymour, and Turcotte \cite{chudnovsky2024cops} proved the following result related to the cops and robbers game.

\begin{quote}
  \textit{If $G$ is connected and $P_5$-free, with $\alpha(G)\ge 3$, then there is a three-vertex induced path of $G$ with vertices $a,b,c$ in order, such that every neighbor of $c$ is also adjacent to one of $a$, $b$.}
\end{quote}

This shows that \cref{conj:stv} holds for $P_5$-free graphs, and furthermore  demonstrates its relation to the cops and robbers game.
We show that the conjecture holds for the $C_5$-free graphs.

In the next two subsections we give the precise definitions and state our main results.

\subsection{Avoidability and inherence of two-rooted graphs}\label{subsec:definitions}

A \emph{two-rooted graph} is a triple $(\hat H,\hat s,\hat t)$ such that $\hat H$ is a graph and $\hat s$ and $\hat t$ are two (not necessarily distinct) vertices of $\hat H$. 
For convenience, and without loss of generality, we will always assume for a two-rooted graph $(\hat H,\hat s,\hat t)$ that $d_{\hat H}(\hat s)\leq d_{\hat H}(\hat t)$.\footnote{Standard definitions from graph theory will be given in~\Cref{preliminaries}.}

Given a two-rooted graph $(\hat H, \hat s, \hat t)$, we refer to $\hat s$ and $\hat t$ as the \emph{$s$-vertex} and the \emph{$t$-vertex} of $(\hat H,\hat s,\hat t)$, respectively.
Given a graph $G$ and a two-rooted graph $(\hat H, \hat s, \hat t)$, a \emph{copy of $(\hat H, \hat s, \hat t)$ in $G$} is any two-rooted graph $(H,s,t)$ such that $H$ is an induced subgraph of~$G$ for which there exists an isomorphism of $\hat H$ to $H$ mapping $\hat s$ to $s$ and $\hat t$ to $t$.

Given a graph $G$, a two-rooted graph $(\hat H, \hat s, \hat t)$, and a copy $(H,s,t)$ of it in $G$, an \emph{extension of $(H, s, t)$ in $G$} is any two-rooted graph $(H',s',t')$ such that $H'$ is an induced subgraph of $G$ obtained from $H$ by adding to it two pendant edges $ss'$ and $tt'$.
In other words, $V(H') = V(H)\cup\{s',t'\}$, vertices $s'$ and $t'$ are distinct, the graph obtained from $H'$ by deleting $s'$ and $t'$ is $H$, and $s$ and $t$ are unique neighbors of $s'$ and $t'$ in $H'$, respectively. 
Furthermore, we say that an extension $(H',s',t')$ of $(H,s,t)$ in $G$ is \emph{closable} if there exists an induced $s',t'$-path in $G$ having no vertex in common with $N_G[V(H')]$ except $s'$ and $t'$. 

\begin{defn}\label{def:avoidable}
Let $(H,s,t)$ be a copy of a two-rooted graph $(\hat H, \hat s, \hat t)$ in a graph $G$.
We say that $(H,s,t)$ is \emph{avoidable} (in $G$) if all extensions of $(H,s,t)$ in $G$ are closable.
\end{defn}

In particular, a copy $(H,s,t)$ of $(\hat H, \hat s, \hat t)$ in $G$ that has no extensions is trivially avoidable.
Such a copy $(H,s,t)$ will be called {\em simplicial}.

\medskip
Note that if $\hat H$ is a path and $\hat s$ and $\hat t$ are its endpoints, the above definitions of an extension of a copy of $(\hat H,\hat s, \hat t)$ in a graph $G$, a closable extension, a simplicial copy, and an avoidable copy coincide with the corresponding definitions for paths as used by Beisegel et al.~\cite{BCGMS19DAM}, Bonamy et al.~\cite{MR4245221}, and Gurvich et al.~\cite{GKMV20+}, in agreement with Chv\'atal et al.~\cite{MR1927566}.
In particular, the definitions of simplicial and avoidable copies also generalize the definitions of simplicial and avoidable vertices in a graph.

\begin{rem}\label{rem:s't'}
The reader may wonder why, in the definition of extension, vertices $s'$ and $t'$ are not allowed to be adjacent in $G$. 
In fact, even if they were, it would not affect avoidability, as any such extension would be trivially closable. 
We keep the above definition in order to be consistent with previous works~\cite{BCGMS19,BCGMS19DAM,MR4245221,GKMV20+,MR1927566}.
\end{rem}

\medskip

Finally, we introduce the core definition of this paper.

\begin{defn}\label{inherent}
A two-rooted graph $(\hat H, \hat s, \hat t)$ is \emph{inherent} if every graph $G$ that contains a copy of $\hat H$ also contains an avoidable copy of $(\hat H, \hat s, \hat t)$.
\end{defn}

The following concept provides a natural certificate of non-inherence.

\begin{defn}\label{def-confining-graph}
Given a two-rooted graph $(\hat H,\hat s,\hat t)$ and a graph $G$, we say that $G$ \emph{confines} $(\hat H,\hat s,\hat t)$ (or: is a \emph{confining graph} for $(\hat H,\hat s,\hat t)$) if $G$ contains a copy of $\hat H$ but no avoidable copy of $(\hat H,\hat s,\hat t)$.
\end{defn}

\subsection{Results}

Using the pendant extension method, we develop the following necessary condition for inherence.

\begin{sloppypar}
\begin{defn}
A \emph{subcubic two-rooted tree} is a two-rooted graph $(\hat H,\hat s,\hat t)$ such that  $\hat H$ is either 
\begin{itemize}
    \item a path with an endpoint $\hat s$, or
    \item a tree with maximum degree~$3$ such that $1 = d_{\hat H}(\hat s) \le d_{\hat H}(\hat t)\le 2$, with $d_{\hat H}(\hat t) = 1$ if and only if $\hat s =\hat  t$.
\end{itemize}
\end{defn}
\end{sloppypar}

\begin{thm}\label{th:tree-deg-conditions}
Every inherent connected two-rooted graph is a subcubic two-rooted tree.
\end{thm}

In addition to this result, we list a large collection of non-inherent subcubic two-rooted trees in \cref{sec:killing-list}.

For paths we have the following results.
By \cref{th:tree-deg-conditions}  a two-rooted path $(\hat  H, \hat s,\hat t)$ is not inherent  if both  $\hat  s$ and $\hat  t$ are internal vertices of the path. 
In particular, not every two-rooted path is inherent.
We restrict further the family of inherent two-rooted paths (see also Figure \ref{pic:th1.8}).

\begin{thm}\label{th:not-inher}
Let $P=(v_0,\dots ,v_{\ell})$ be a path, $s = v_0$, and $t\in V(P)$.
Then, the two-rooted path $(P,s,t)$ is not inherent if one of the following conditions holds:
\begin{enumerate}
    \item[(i)]  $\ell\ge 1$ and $t=v_0$,
    \item[(ii)] $\ell\ge 3$ and $t=v_{\ell-1}$, 
    \item[(iii)] $\ell = 3$ and $t = v_1$,
    \item[(iv)]  $\ell=4$ and $t=v_1$,
    \item[(v)]   $\ell=5$ and $t=v_2$.
\end{enumerate}
\end{thm}

\begin{figure}[t]
\centering
\begin{tabular}{c@{\hspace*{32pt}}c}
  \mpfile{more-graphs}{101}&\mpfile{more-graphs}{102}\\
   (i) & (ii) 
\end{tabular}
\vskip4mm
\begin{tabular}{c@{\hspace*{24pt}}c@{\hspace*{24pt}}c}
  \mpfile{more-graphs}{103}&\mpfile{more-graphs}{104}&\mpfile{more-graphs}{105}\\
   (iii) & (iv) &  (v)
\end{tabular}

\caption{Two-rooted paths for Theorem~\ref{th:not-inher}}\label{pic:th1.8}
\end{figure}

As already mentioned, the \cref{conj:stv} holds for $P_5$-free graphs \cite{chudnovsky2024cops}. We provide the following result.

\begin{thm}\label{th:stv}
Assume that graph $G$ contains an induced copy of the two-edge path $P=(s,t,v)$ with roots $s$ and $t$.
Then $G$ also contains an avoidable copy of it whenever at least one of the following two
conditions holds:
\begin{enumerate}[a)]
     \item  $G$ has no induced $C_5$,\label{th:stv-C5-free}
     
    \item  all vertices of $G$ have degree at most $3$.\label{th:stv-cubic}
\end{enumerate}
\end{thm}
In fact, condition~\ref{th:stv-C5-free}) of \cref{th:stv} can be weakened as follows: $G$ has no induced subgraph isomorphic to $C_5+K_1$ (that is, the graph obtained from $C_5$ by adding to it an isolated vertex).

In \cref{prop:disconnected} we also obtain a sufficient condition for inherence of disconnected two-rooted graphs.

\subsection*{Structure of the paper}

In \cref{preliminaries}, we provide some definitions needed throughout the paper. 
In \cref{sec:pendant-extensions}, we introduce the method of pendant extension, which leads to the proof of \cref{th:tree-deg-conditions}.
In addition to this result, we list in \cref{sec:killing-list} a large collection of non-inherent subcubic two-rooted trees, which cannot be confined by the method of pendant extensions.
This leads to the proof of \cref{th:not-inher}.
In \cref{sec:positive-results}, we provide a more general approach that allows us to provide an alternative proof of  \cref{th:endpoint-rooted paths are inherent}, as well as to prove \cref{th:stv}.
We also list some other open cases and give a sufficient condition for inherence of disconnected two-rooted graphs.
In \cref{sec:open-questions}, we discuss some open questions and possible generalizations.

\section{Preliminaries}\label{preliminaries}

For the reader's convenience we reproduce some standard definitions from graph theory which will be used in the paper.

All graphs considered in this paper are simple and undirected.
They are also finite unless explicitly stated otherwise, and this is only relevant in \Cref{sec:pendant-extensions}.
The \emph{order} of a graph is the cardinality of its vertex set.
For a graph $G$ and a set $X\subseteq V(G)$, we denote by $N_G(X)$ the \emph{(open) neighborhood} of $X$, that is, the set of vertices in $V(G)\setminus X$ that are adjacent to a vertex in $X$, and by 
$N_G[X]$ the \emph{closed neighborhood} of $X$, that is, the set $X\cup N_G(X)$.
Two vertices $u$ and $v$ are said to be \emph{twins} if $N(u)\setminus \{v\} = N(v)\setminus \{u\}$; if in addition they are adjacent, then they are said to be \emph{true twins}, and \emph{false twins} if they are non-adjacent.

A \emph{clique} in a graph $G=(V,E)$ is a set of vertices $C\subseteq V$ such that $uv\in E$ for any $u,v\in C$. 
A \emph{path} is a graph with vertex set $\{v_0,v_1,\ldots, v_\ell\}$ in which two vertices $v_i$ and $v_j$ with $i<j$ are adjacent if and only if $j = i+1$; the vertices $v_0$ and $v_\ell$ are the \emph{endpoints} of the path.
We sometimes denote such a path simply by the sequence $(v_0,\ldots, v_\ell)$.
The \emph{length} of a path is defined as the number of its edges.
We denote a $k$-vertex path by $P_k$.
A \emph{cycle} is a graph obtained from a path of length at least three by identifying vertices $v_0$ and $v_\ell$. 
The \emph{girth} of a graph $G$ is the minimal length of a cycle in it (or $\infty$ if $G$ is acyclic).
It is known that for all integers $k\ge 2$ and $g\ge 3$, there exists a $k$-regular graph and girth~$g$ (see~\cite{MR165515}, as well as~\cite{MR4336218}).
Such a graph with smallest possible order is called a {\it $(k,g)$-cage}.

An \emph{isomorphism} from a graph $G_1=(V_1,E_1)$ to a graph $G_2=(V_2,E_2)$ is a bijection $f:V_1\to V_2$ such that $uv\in E_1$ if and only if $f(u)f(v)\in E_2$. 
An \emph{isomorphism of two-rooted graphs} $(G_1,s_1,t_1)$ and $(G_2,s_2,t_2)$ is an isomorphism $f$ from $G_1$ to $G_2$ such that $f(s_1)=s_2$ and $f(t_1)=t_2$.
An \emph{automorphism} of a graph $G$ is an isomorphism of $G$ to itself.

Let  $G=(V,E)$ be a graph, and let $S\subseteq V$ be any subset of vertices of $G$. 
We define the \emph{induced subgraph} $G[S]$ to be the graph with vertex set $S$ whose edge set consists of all of the edges in $E$ that have both endpoints in $S$. 
In this paper, all subgraphs are assumed to be induced unless explicitly indicated otherwise.
Given two graphs $G$ and $H$, a \emph{copy of $H$ in $G$} is a subgraph of $G$ isomorphic to $H$.
We say that the graph $G$ is \emph{$H$-free} if it does not admit any copy of $H$. 
For a collection of graphs $\mathcal H$ we say that $G$ is \emph{$\mathcal H$-free} if it does not admit any copy of $H$ for any $H\in \mathcal H$. 

A \emph{bridge} is an edge of a graph whose deletion increases the number of connected components. 
All other edges are \emph{non-bridges}.

The degree of a vertex $v\in V(G)$ is denoted by $d_G(v)$. 
We denote the maximum degree of a graph $G$ by $\Delta(G)$.
If $d_G(v)=1$ then we say that $v$ is a \emph{pendant vertex}, or a \emph{leaf}. 
If $\Delta(G)\leq 3$ then we say that $G$ is \emph{subcubic}. 
A graph is \emph{cubic} if all its vertices have degree $3$.
As usual, the \emph{distance} between two vertices $u$ and $v$ in a connected graph $G$ is the length of a shortest $u,v$-path; it is denoted by $\dist_G(u,v)$.
The \emph{disjoint union of graphs} $G_1=(V_1,E_1)$ and $G_2=(V_2,E_2)$, where $V_1\cap V_2 = \emptyset$, is defined as a graph $G_1+G_2=(V_1\cup V_2,E_1\cup E_2)$.
A \emph{rooted tree} is a pair $(T,r)$ where $T$ is a tree and $r\in V(T)$.
A \emph{rooted forest} is a disjoint union of rooted trees.
We define the \emph{depth} of a rooted tree $(T,r)$ as the eccentricity of its root $r$, i.e., $\max_{v\in V(T)}\dist_T(r,v)$. 
Correspondingly, the \emph{depth} of a rooted forest is defined as the maximal depth among all its rooted trees. 

The operation of \emph{subdividing} an edge $uv$ in a graph $G=(V,E)$ results in a graph $G'=(V\cup\{w\},E')$ such that $w\notin V$ and $E'=(E\setminus\{uv\})\cup\{uw,wv\}$.
The \emph{lexicographic product} of graphs $G$ and $H$ (see, e.g.,~\cite{MR2817074}) is the graph  $G[H]$ such that
\begin{itemize}
    \item the vertex set of $G[H]$ is $V(G) \times V(H)$; and
    \item any two vertices $(u,v)$ and $(x,y)$ are adjacent in $G[H]$ if and only if either $u$ is adjacent to $x$ in $G$, or $u = x$ and $v$ is adjacent to $y$ in $H$.
\end{itemize}

\section{The method of pendant extensions}\label{sec:pendant-extensions}

In this section we mainly consider connected two-rooted graphs, but some results are valid without this assumption. 
We show that all the inherent graphs are forests.
To prove this, we introduce the following definitions.


%
\begin{defn}[Pendant Extension]\label{def:PE}
Given a graph $G$, a two-rooted graph $(\hat H, \hat s, \hat t)$, and a simplicial copy $(H,s,t)$ of $(\hat H, \hat s, \hat t)$ in $G$, a \emph{pendant extension (PE)} of $G$ (with respect to $(H,s,t)$) is any graph $G'$ obtained from $G$ by adding to it the minimal number of pendant edges to $s$ and/or $t$ so that $(H,s,t)$ becomes non-simplicial in $G'$.
\end{defn}

\begin{defn}[PE-Sequence]\label{def:seq_PE}
A \emph{PE-sequence} of a two-rooted graph $(\hat H, \hat s, \hat t)$ is an arbitrary sequence, finite or infinite, of graphs $(G_i)_{i\ge 0}$, obtained recursively as follows. 
Initialize $G_0=\hat H$.
For $i\geq 0$, if $G_{i}$ contains a simplicial copy $(H,s,t)$ of $(\hat H, \hat s, \hat t)$, then the next graph in the sequence is any graph $G_{i+1}$ that is a PE of $G_i$ with respect to $(H,s,t)$.
Otherwise, the sequence is finite, having $G_i$ as the final graph.
\end{defn}

\begin{example}
The two-rooted graph $(P_1,s,t)$, where $V(P_1) = \{s\} = \{t\}$, has an infinite PE-sequence (see \Cref{pic:P1ext}). In the
figures we mark the vertices of $G_i$ by black and the pendant
vertices added to $G_i$ by white. 
\begin{figure}[!h]
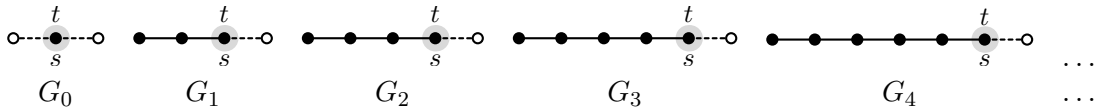

\centering
\begin{tabular}{c@{\quad}c@{\quad}c@{\quad}c@{\quad}c@{\quad}c}
\mpfile{ext}{100}&\mpfile{ext}{101}& \mpfile{ext}{102}& \mpfile{ext}{103}& \mpfile{ext}{104}& $\dots$
\\
$G_0$ & $G_1$ & $G_2$ & $G_3$ & $G_4$&\dots
\end{tabular}
\caption{An infinite PE-sequence of $(P_1,s,t)$, where $V(P_1) = \{s\} = \{t\}$}
  \label{pic:P1ext}
\end{figure}

Contrary, for the two-rooted graph formed by the cycle $C_3$ with the two roots $s$ and $t$ adjacent, all PE sequences are finite (for an example, see \Cref{pic:C3ext}).

\begin{figure}[!h]
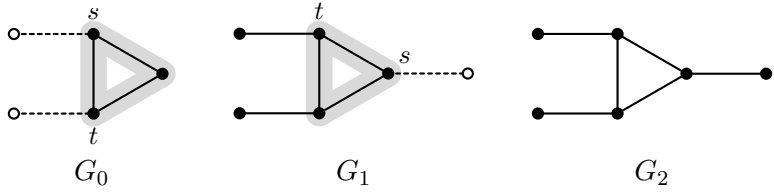

\centering
\begin{tabular}{c@{\qquad}c@{\qquad}c}
  \mpfile{ext}{20}&\raisebox{6.78pt}{\mpfile{ext}{21}}&
  \raisebox{9.39pt}{\mpfile{ext}{22}}
\\
$G_0$ & $G_1$ & $G_2$ 
\end{tabular}
  \caption{A finite PE-sequence of $(C_3,s,t)$, $s\ne t$}\label{pic:C3ext}
\end{figure}
\end{example}

By construction, every PE-sequence has the following property, which we refer to as the {\it No New Cycle Property}.

\begin{lem}[No New Cycle Property]\label{NoNewCycleProperty}
For every two-rooted graph $(\hat H, \hat s, \hat t)$, each PE-sequence $(G_i)_{i\ge 0}$ of $(\hat H, \hat s, \hat t)$, and all $i\ge 0$, the graph $G_i$ is obtained from the graph $G_0 = \hat H$ by adding to it some pendant trees.
Consequently, every cycle in $G_i$ is contained in every copy of $\hat H$ in $G_i$.\hfill\qed
\end{lem}

Notice that a two-rooted graph $(H,s,t)$ can have many different PE-sequences.

\begin{example}
  The two-rooted graph $(C_4, s,t)$, where $s$ and $t$ are adjacent, has PE-sequences of different lengths (see
  \Cref{pic:C4ext}). Note that the final graphs of the two sequences
  coincide. This is not a coincidence (see~\cref{thm:unique-proper-extension}). 
\begin{figure}[!h]
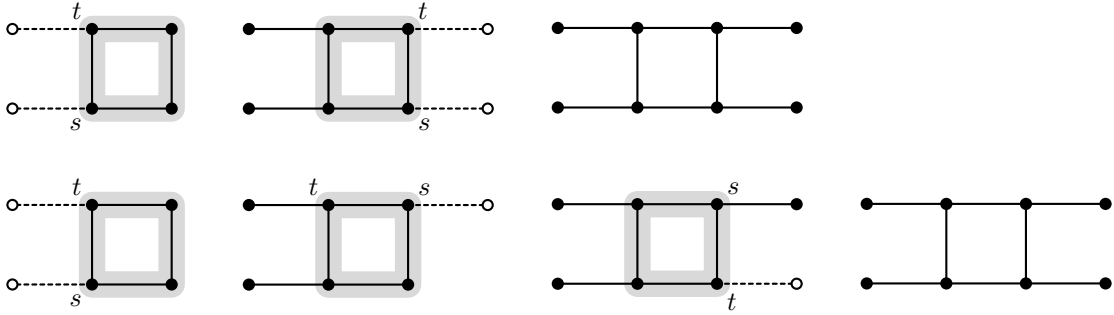

\centering
\begin{tabular}{c@{\qquad}c@{\qquad}c@{\qquad}c}
\mpfile{ext}{30}&\mpfile{ext}{31}&\raisebox{6.26pt}{\mpfile{ext}{32}}&
\\[5mm]
\mpfile{ext}{30}&\raisebox{3.13pt}{\mpfile{ext}{33}}
&\raisebox{-1.56pt}{\mpfile{ext}{34}}&\raisebox{6.26pt}{\mpfile{ext}{32}}
\end{tabular}
  \caption{Two PE-sequences for $(C_4,s,t)$, where $s$ and $t$ are adjacent}\label{pic:C4ext}
\end{figure}
\end{example}

\begin{defn}
A sequence, finite or infinite, $S=(G_0,G_1,\ldots)$ of graphs, is said to be \emph{non-decreasing} if for all
$G_i\in S$, $i>0$,  the graph $G_{i-1}$ is a subgraph of $G_{i}$. 
The \emph{limit graph} of a non-decreasing sequence $S=(G_0,G_1,\ldots)$ of graphs is the (finite or infinite) graph $G(S)$ such that 
\[
    V(G(S))=\bigcup_{i\ge 0} V(G_i) \quad\text{and}\quad
    E(G(S))=\bigcup_{i\ge 0} E(G_i).
\]
\end{defn}

Let us note that given a two-rooted graph $(\hat H,\hat s,\hat t)$, the limit graphs of two PE-sequences of $(\hat H,\hat s,\hat t)$ need not be isomorphic. 
Moreover, one of them may not contain any simplicial copies of $(\hat H,\hat s,\hat t)$, while the other may.

\begin{example}
Two nonisomorphic limit graphs for the two-rooted graph $(P_1,s,t)$,
where $s=t$, are shown in \Cref{pic:P1two}. The first limit graph contains simplicial copies of $(P_1,s,t)$, while the
second does not.
\begin{figure}[!h]
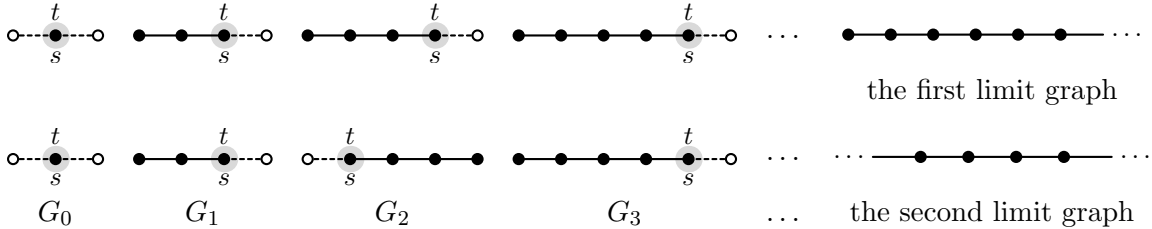

\centering
\begin{tabular}{c@{\quad}c@{\quad}c@{\quad}c@{\quad}c@{\quad}c}
  \mpfile{ext}{100}&\mpfile{ext}{101}& \mpfile{ext}{102}& \mpfile{ext}{103}&
  \raisebox{9.39pt}{$\dots$}&\raisebox{8.34pt}{\mpfile{ext}{105}}
  \\
  &&&&&the first limit graph\\[2.5mm]
  \mpfile{ext}{10}&\mpfile{ext}{11}& \mpfile{ext}{12}& \mpfile{ext}{13}&
  \raisebox{9.39pt}{$\dots$}&\raisebox{8.34pt}{\mpfile{ext}{15}}
\\
$G_0$ & $G_1$ & $G_2$ & $G_3$ &\dots&
the second limit graph
\\
\end{tabular}
  \caption{Two PE-sequences of $(P_1,s,t)$, where $V(P_1) = \{s\} = \{t\}$}\label{pic:P1two}
\end{figure}
\end{example}

\begin{defn}\label{DefnFinitelyExtendable}
A two-rooted graph $(\hat H,\hat s,\hat t)$ is \emph{PE-confined} if there exists a finite PE-sequence of $(H,s,t)$. 
\end{defn}

\begin{prop}\label{TreelyNonInherent}
Every PE-confined two-rooted graph is non-inherent. 
\end{prop}

\begin{proof}
Let $(\hat H,\hat s,\hat t)$ be a PE-confined two-rooted graph.
Fix a finite PE-sequence $(G_i)_{i\ge 0}$ of  $(\hat H,\hat s,\hat t)$. 
The limit graph $G$ of this sequence is a finite graph.
As we show next, $G$ certifies that $(\hat H,\hat s,\hat t)$ is not inherent.
Since $G_0 = \hat H$ and $G_0$ is a subgraph of $G$, we infer that $G$ contains a copy of $\hat H$.
On the other hand, if $(H,s,t)$ is an arbitrary copy of $(\hat H,\hat s,\hat t)$ in $G$, then we show that this copy is not avoidable.
Since $G$ is the limit graph of a finite PE-sequence of $(\hat H,\hat s,\hat t)$, the copy $(H,s,t)$ is not simplicial in $G$.
Therefore, it has an extension $(H',s',t')$.
By the No New Cycle Property, this extension is not closable.
Hence, the copy $(H,s,t)$ is not avoidable.
We conclude that $G$ confines $(\hat H,\hat s,\hat t)$.
Consequently, the two-rooted graph $(\hat H,\hat s,\hat t)$ is non-inherent.
\end{proof}

By \Cref{TreelyNonInherent}, a necessary condition for a two-rooted graph to be inherent is that all its PE-sequences are infinite. 
This motivates the next definition.

\begin{defn}\label{DefnPotentiallyInherent}
A two-rooted graph $(\hat H,\hat s,\hat t)$ is  \emph{PE-inherent} if all PE-sequences of $(\hat H,\hat s,\hat t)$ are infinite.
\end{defn}

By \Cref{DefnFinitelyExtendable,DefnPotentiallyInherent}, every two-rooted graph is either PE-confined or PE-inherent (see examples in \Cref{pic:P1ext,pic:C3ext}).
As a consequence of \Cref{DefnFinitelyExtendable,TreelyNonInherent,DefnPotentiallyInherent}, we obtain the following.

\begin{cor}\label{inherent-implies-PE-inherent}
   Every inherent two-rooted graph is PE-inherent.
\end{cor}

However, a PE-inherent two-rooted graph may be inherent or not (see examples in \cref{sec:killing-list}).

\medskip
We will show in~\Cref{subsection All proper sequences have the same limit} that no connected two-rooted graph $(\hat H,\hat s,\hat t)$ can have both a finite and an infinite PE-sequence.
Thus, we can simply consider an arbitrary PE-sequence to determine whether $(\hat H,\hat s,\hat t)$ is PE-confined or PE-inherent.

\begin{defn}\label{DefnProper}
A PE-sequence $S$ of a two-rooted graph $(\hat H,\hat s,\hat t)$ is called \emph{proper} if the limit graph $G(S)$ has no simplicial copies of $(\hat H,\hat s,\hat t)$.
\end{defn}

Clearly, if a PE-sequence is finite, it is proper.
Our next section deals with uniqueness of the limit graph for proper PE-sequences of a connected two-rooted graph in both finite and infinite cases. 
This is of independent interest. For our purposes, it would be enough just to distinguish whether all PE-sequences are finite (in which case $(\hat H,\hat s,\hat t)$ is PE-confined) or whether they are all infinite (in which case $(\hat H,\hat s,\hat t)$ is PE-inherent). 

\subsection{All proper PE-sequences have the same limit}\label{subsection All proper sequences have the same limit}

The purpose of this section is to prove the following theorem.

\begin{thm}\label{thm:unique-proper-extension}
Every two-rooted graph $(\hat H, \hat s, \hat t)$ admits a proper PE-sequence. Moreover, for an arbitrary PE-sequence $S$ of $(\hat H, \hat s, \hat t)$, 
its limit graph $G(S)$ is 
a subgraph of 
the limit graph of some proper PE-sequence of $(\hat H, \hat s, \hat t)$.
Furthermore, if $\hat H$ is connected, then the limit graphs of all proper PE-sequences of $(\hat H,\hat s,\hat t)$ are isomorphic to each other.
\end{thm}

Before we prove the theorem, let us comment on its importance.
Due to the theorem, we are able to define the \emph{limit graph} $G(\hat H,\hat s,\hat t)$ of a connected two-rooted graph $(\hat H,\hat s,\hat t)$ as the limit graph of any proper PE-sequence of $(\hat H,\hat s,\hat t)$. 
Whenever this limit graph $G(\hat H,\hat s,\hat t)$ is finite, it confines $(\hat H,\hat s,\hat t)$, as can be seen from the proof of \cref{TreelyNonInherent}.
Furthermore, as we show in \cref{prop:PE-confining}, the limit graph $G(\hat H,\hat s,\hat t)$, if finite, is a subgraph of any confining graph for $(\hat H,\hat s,\hat t)$.

The proof of the theorem is based on the concept of a stage sequence, defined as follows. 

\begin{defn}\label{stage-sequence}
A \emph{stage sequence} of a two-rooted graph $(\hat H,\hat s,\hat t)$ is any PE-sequence $S = (G_0,G_1,\ldots)$, finite or infinite, of $(\hat H,\hat s,\hat t)$ obtained in countably many \emph{stages}, as follows. 
We set $G_0 =\hat  H$ and $S = (G_0)$; this is the output of stage~$0$. For every $j\ge 0$, the output of stage $j$ is the input to stage $j+1$. 
The output of stage $j$ is a sequence of the form $S = (G_0,\ldots, G_{k_{j}})$. 
Note that $k_0=0$.
Stage $j+1$ works as follows.
If $G_{k_{j}}$ does not contain any simplicial copies of $(\hat H,\hat s,\hat t)$, then we
output the current sequence $S$.
Otherwise, let $\mathcal{H}_{j}$ be the set of all simplicial copies of $(\hat H,\hat s,\hat t)$ in $G_{k_{j}}$ and set $G' = G_{k_{j}}$. 
We iterate over all two-rooted graphs in $\mathcal{H}_{j}$ and keep updating the graph $G'$ by extending the current copy of $(\hat H,\hat s,\hat t)$, whenever necessary (as in~\cref{def:seq_PE}), so that in the end no copy from $\mathcal{H}_{j}$ is simplicial in $G'$.
Each intermediate version of $G'$ is appended to the current sequence~$S$.
This completes the description of stage $j+1$ and, thus, of the construction of a stage sequence $S$ as well. 
The number $k_j$ will be referred to as the \emph{stage $j$ index of $S$} and the graph $G_{k_j}$ as the \emph{stage $j$ graph of $S$}. 
\end{defn}

\begin{lem}\label{lm:stage-is-proper}
Any stage sequence of a two-rooted graph $(\hat H,\hat s,\hat t)$ is a proper PE-sequence of $(\hat H,\hat s,\hat t)$. 
\end{lem}

\begin{proof}
Arguing by contradiction, suppose that $S$ is a stage sequence of a two-rooted graph $(\hat H,\hat s,\hat t)$ such that the limit graph $G(S)$ has a simplicial copy $(H,s,t)$ of $(\hat H,\hat s,\hat t)$.
Then $(H,s,t)$ is also a simplicial copy of $(\hat H,\hat s,\hat t)$ in the stage $j$ graph of $S$ for some $j$.
Thus, the stage $j+1$ graph of $S$ exists and $(H,s,t)$ is not simplicial in it.
Since the stage $j+1$ graph is a subgraph of the limit graph $G(S)$, 
this copy is also not simplicial in $G(S)$, a contradiction.
\end{proof}

An example of a stage sequence of the two-rooted graph $(P_1, s,t)$ with $s = t$ is shown in \cref{pic:P1two} (the second sequence), where $k_j=2j-1$ for $j>0$.

\begin{lem}\label{lemma:stage-ext}
For every connected two-rooted graph $(\hat H,\hat s,\hat t)$, every two stage sequences $S$ and $S'$ of $(\hat H,\hat s,\hat t)$, and every $j\ge 0$, the stage $j$ graphs of $S$ and $S'$ are isomorphic to each other.
\end{lem}

In view of \cref{lemma:stage-ext}, the notion of a \emph{stage $j$ graph} of a connected two-rooted graph $(\hat H,\hat s,\hat t)$ is well-defined (up to isomorphism) for every non-negative integer $j$, and will stand for the stage $j$ graph of an arbitrary but fixed stage sequence of  $(\hat H,\hat s,\hat t)$.
Before giving a proof of \cref{lemma:stage-ext}, we illustrate it with an example.

\begin{example}
  In \Cref{pic:P3adj} the stage $3$ graph of  $(P_3,s,t)$, where $s$ and $t$
  are adjacent, is shown.
\begin{figure}[!h]
  \centering
  \mpfile{ext}{5}
  \caption{The stage $3$ graph of  $(P_3,s,t)$, where $s$ and $t$ are adjacent. 
  Each vertex is labeled by the minimal $i$ such that the stage $i$ graph contains this vertex.}\label{pic:P3adj}
\end{figure}
\end{example}

We prove \cref{lemma:stage-ext} by proving a more detailed statement \cref{stage-j-graphs-and-extensions}.
To this end we define a graph transformation that takes as input a two-rooted graph $(\hat H,\hat s,\hat t)$ and a graph $G$ and computes a graph $G'$. 
As will be shown in the proof of~\cref{stage-j-graphs-and-extensions}, given a stage sequence $S$ of a two-rooted graph $(\hat H,\hat s,\hat t)$ and a stage $j$ graph $G$ of $S$, the resulting graph $G'$ is the stage $j+1$ graph of $S$. 
This implies the claimed uniqueness of the limit graph of any stage sequence of a two-rooted graph $(\hat H,\hat s,\hat t)$.

\begin{defn}\label{Hst-extension}
Let $G$ be a graph and let $(\hat H,\hat s,\hat t)$ be a two-rooted graph.
The \emph{$(\hat H,\hat s,\hat t)$-extension} of a graph $G$ is the graph $G'$ obtained from $G$ as follows.
\begin{itemize}
\item If $\hat s \neq \hat t$, we add one pendant edge to each vertex $v\in V(G)$ such that there exists a copy $(H,s,t)$ of $(\hat H,\hat s,\hat t)$ in $G$ such that $v\in \{s,t\}$ and $d_{G}(v) = d_{H}(v)$.\footnote{Note that such a copy is necessarily simplicial.} 

\item If $\hat s = \hat t$, then to each vertex $v\in V(G)$ such that there exists a copy $(H,s,t)$ of $(\hat H,\hat s,\hat t)$ in $G$ such that $v = s$, we add
\begin{itemize}
    \item one pendant edge if $d_{G}(v) = d_{H}(v) +1$;
    \item two pendant edges if $d_{G}(v) = d_H(v)$.
\end{itemize}
In other words, $G'$ is the graph obtained from $G$ by adding  exactly $2-(d_{G}(v)-d_{H}(v))$ pendant edges to $v$.
\end{itemize}
\end{defn}

\begin{rem}
Let $G$ be any stage graph of a stage sequence of $(\hat H,\hat s,\hat t)$ and let $G'$ be the $(\hat H,\hat s,\hat t)$-extension of  $G$. 
Then no copy of $(\hat H,\hat s,\hat t)$ in $G$ is simplicial in $G'$.
However, when $G$ is an arbitrary graph containing a copy of $\hat H$, this is not necessarily the case.
For example, take a graph $G$ and a two-rooted graph $(\hat H,\hat s,\hat t)$ such that $G$ is isomorphic to $C_4$ and $(\hat H,\hat s,\hat t)$ is isomorphic to the endpoint-rooted path $P_2$. 
In this case the $(\hat H,\hat s,\hat t)$-extension of $G$ coincides with $G$, although $C_4$ contains simplicial copies of $(\hat H,\hat s,\hat t)$.
\end{rem}

\cref{lemma:stage-ext} is an immediate consequence of the following lemma.

\begin{lem}\label{stage-j-graphs-and-extensions}
Let $S = (G_0,G_1,\ldots)$ be a stage sequence of a connected two-rooted graph $(\hat H,\hat s,\hat t)$.
Then for any $j\ge 1$, the stage $j$ graph of $S$ coincides with the \hbox{$(\hat H,\hat s,\hat t)$-extension} of the stage $j-1$ graph of $S$.
\end{lem}

\begin{proof}
Fix $j\ge 1$. Let us denote by $G$ the stage $j-1$ graph of $S$, by $G'$ the $(\hat H,\hat s,\hat t)$-extension of $G$, and by  $G''$ the stage $j$ graph of $S$.

We obtain $G''$ from $G$ by processing in some order all simplicial copies of $(\hat H,\hat s,\hat t)$ in $G$. 
We show that $G''$ is isomorphic to $G'$.
Note that both graphs $G''$ and $G'$ are obtained from $G$ by adding some non-negative number of pendant edges to each vertex. For each vertex $v\in V(G)$, let us denote by
\begin{itemize}
\item $f_1(v)$ the number of pendant edges added to  $v$ when constructing 
 $G'$ from $G$, that is, $f_1(v) = d_{G'}(v) -d_{G}(v)$;
     \item $f_2(v)$ the number of pendant edges added to  $v$ when constructing
 $G''$ from $G$, that is, $f_2(v) = d_{G''}(v) - d_{G}(v)$.
\end{itemize}
  
Note that for all $v\in V(G)$, we have
$$\begin{array}{ll}
f_1(v)\in \{0,1\}, & \hbox{if $\hat s\neq \hat t$;} \\
f_1(v)\in \{0,1,2\}, & \hbox{if $\hat s = \hat t$.}
  \end{array}$$

To prove that $G''$ and $G'$ are isomorphic, it suffices to show that $f_2(v) = f_1(v)$ for all $v\in V(G)$. Suppose for a contradiction that $f_2(v)\neq f_1(v)$ for some $v\in V(G)$.
Consider first the case when $f_2(v) < f_1(v)$. Since $f_1(v)>0$, there exists a simplicial copy $(H,s,t)$ of $(\hat H,\hat s,\hat t)$ in $G$ such that $v = s$ (or $v = t$) and that gives rise to $f_1(v)$ new pendant edges at $v$. 
By analyzing the cases (i) $\hat s\neq \hat t$, (ii) $\hat s = \hat t$ and $d_{G}(v) = d_H(v)$, and (iii) $\hat s = \hat t$ and $d_{G}(v) \neq d_H(v)$, it is not difficult to verify that, since the number of added edges to $v$ in $G''$ from $G$ is smaller than $f_1(v)$, 
the copy $(H,s,t)$ still remains simplicial in $G''$, a contradiction. 
This shows that $f_2 \geq f_1$, that is, $f_2(v) \geq f_1(v)$ for all $v\in V(G)$. 

Suppose now that $f_2(v) > f_1(v)$. 
Recall that $k_{j-1}$ and $k_j$ denote the stage $j-1$ and stage $j$ indices of $S$, respectively.
Since in the process of transforming $G = G_{k_{j-1}}$ to $G''=G_{k_{j}}$, pendant edges are added to $v$, there exists a minimal integer $\ell \in \{k_{j-1},k_{j-1}+1,\ldots, k_{j}-1\}$ such that $d_{G_{\ell+1}}(v) >d_{G_{\ell}}(v)+f_1(v)$. 
This means that $G_{\ell+1}$ is produced from $G_\ell$ by adding more than $f_1(v)$ pendant edges to the vertex $v$, where $v = s$ or $v = t$ for some copy $(H,s,t)$ of the two-rooted graph $(\hat H,\hat s,\hat t)$ in $G$.
We present the proof for the case $v = s$; the arguments for the case $v = t$ are the same.
We claim that only $f_1(v)$ edges can be used in the extension considered and come to a contradiction with the minimality of the number of pendant edges in a PE. 

Consider the possible cases. 
\begin{enumerate}[(A)]
    \item $f_1(v)=2$. This implies that $\hat s= \hat t$. 
    Due to the minimality requirement of \cref{def:PE}, at most two pendant edges are added, a contradiction. 
    \item $f_1(v)=1$, $\hat s\ne \hat t$.  In this case, at most one pendant edge is added to the vertex in the PE, a contradiction. 
    \item $f_1(v) = 0$, $\hat s\ne \hat t$. 
    We must have $d_{G}(v)>d_{H}(v)$, since the equality $d_{G}(v)=d_{H}(v)$ would imply that $f_1(v)>0$, by the definition of $G'$.
    In particular, there exists an edge $vw$ from $E(G)\setminus E(H)$ incident with $v$.
    The graph $G_{\ell+1}$ is obtained from the graph $G_\ell$ by extending $(H,s,t)$ to some two-rooted graph $(H',s',t')$.
    We claim that $s'\in V(G)$. Suppose that this is not the case.
    Consider the two-rooted graph $(H'', w, t')$ such that $H''$ is the subgraph of $G_{\ell+1}$ induced by $(V(H')\setminus\{s'\})\cup\{w\}$.
    We claim that $(H'', w, t')$ is an extension of $(H,s,t)$ in $G_{\ell+1}$.
    First, note that $sw = vw$ is an edge in $G_{\ell+1}$.
    Furthermore, the copy $H$ of $\hat H$ contains all cycles of $G_{\ell+1}$ due to \cref{NoNewCycleProperty} (No New Cycle Property).
    In particular, since $H$ is connected, this implies that $N_{G_{\ell+1}}(w)\cap (V(H)\cup\{t'\}) = \{v\}$.
    Thus, $(H'', w, t')$ is indeed an extension of $(H,s,t)$ in $G_{\ell+1}$, as claimed.
    However, this contradicts the minimality requirement from the definition of a PE-sequence (with respect to computing $G_{\ell+1}$ from $G_\ell$). 
    This shows that $s'\in V(G)$.
    Consequently, $d_{G_{\ell+1}}(v) = d_{G_{\ell}}(v)$, which contradicts the inequality 
    $d_{G_{\ell+1}}(v) >d_{G_{\ell}}(v)+f_1(v)$. \label{it:C}
    
    \item $f_1(v)=1$, $\hat s= \hat t$. This implies that $d_{H}(v)=d_{G}(v)-1$ and the argument is similar to that of case \eqref{it:C}. 
    
    \item $f_1(v)=0$, $\hat s= \hat t$. This implies that $d_{H}(v)=d_{G}(v)$ and the argument is similar to that of case \eqref{it:C}.\qedhere
\end{enumerate}
\end{proof}
Thus \cref{lemma:stage-ext} is proved, and we are now ready to prove \cref{thm:unique-proper-extension}.
\begin{proof}[Proof of \cref{thm:unique-proper-extension}]
Fix a two-rooted graph $(\hat H,\hat s,\hat t)$. Any stage sequence $S^*$ of $(\hat H,\hat s,\hat t)$ is proper by~\cref{lm:stage-is-proper}.

Consider an arbitrary PE-sequence $S = (G_0,G_1,\ldots)$ of $(\hat H,\hat s,\hat t)$. 
We show how to construct a stage sequence $S^*=(G^*_0, G^*_1, \ldots)$ of $(\hat H,\hat s,\hat t)$ such that for all $i\ge 0$ there exists a smallest integer $j(i)\ge 0$ such that $G_i$ is a subgraph of the stage $j(i)$ graph of $S^*$; note that since for all $i\ge 0$, the graph $G_i$ is a subgraph of $G_{i+1}$, such a function $i\mapsto j(i)$ will be nondecreasing.
The construction is by induction on $i$.
For $i = 0$, we set $j(0) = 0$, since the initial graph of any PE-sequence of $(\hat H,\hat s,\hat t)$ is $G^*_0 = \hat H$ by definition, and starting $S^*$ with $G^*_0$ will assure that $G_0$ is the stage $0$ graph of $S^*$.
Let now $i\ge 1$ and assume that the function $i'\mapsto j(i')$ is defined for the range $i'\in \{0,\ldots, i-1\}$ and is nondecreasing.
In particular, at this point, the sequence $S^*$ has already been constructed up to stage $q = \max\{j(i'): 0\le i'\le i-1\}$; note that $q = j(i-1)$ since the function $i'\mapsto j(i')$ is nondecreasing.
Consider the graph $G_i$. 
By the definition of a PE-sequence, $G_{i-1}$ contains a simplicial copy $(H,s,t)$ of $(\hat H,\hat s,\hat t)$ that is not simplicial in $G^*_i$.
If $G_i$ is a subgraph of the stage $q$ graph of $S^*$, then we set $j(i) = q$.
Otherwise, a copy $(H,s,t)$ is simplicial in the stage $q$ graph of $S^*$, and we continue the sequence $S^*$ by considering any stage that begins by extending the simplicial copy $(H,s,t)$ in $G_{i-1}$, which is a subgraph of the stage $q$ graph of $S^*$.
This defines stage $q+1$ of $S^*$.
We set $j(i) = q+1$, as by construction $G_i$ is a subgraph of the stage $q+1$ graph of $S^*$.
This shows the existence of a stage sequence $S^*$ of $(\hat H,\hat s,\hat t)$ such that each $G_i$ is a subgraph of the limit graph $G(S^*)$; equivalently, the limit graph $G(S)$ is a subgraph of $G(S^*)$.

Assume now that $\hat H$ is connected. 
Note that $(\hat H,\hat s,\hat t)$ may have several stage sequences, as they may depend on the order in which the simplicial copies of $(\hat H,\hat s,\hat t)$ in $G_{k_{j}}^*$ are processed within stage $j+1$.
Nevertheless, \cref{lemma:stage-ext} shows that any stage sequence $S^*$ is `stage-wise' unique, i.e., for each $j\ge 0$, the stage $j$ graph of $S^*$ is unique up to isomorphism: it is isomorphic to the $(\hat H,\hat s,\hat t)$-extension of the stage $j-1$ graph of $S^*$.

This implies that the limit graph $G(S^*)$ is unique up to isomorphism.
Indeed, since  the stage sequences are `stage-wise' unique, either they are all finite, with the same number of stages, or they are all infinite. In the former case, there exists a unique positive integer $j$ such that the limit graph of any stage sequence $S^*$ is isomorphic to the stage $j$ graph of $S^*$.
In the latter case, clearly $G(S^*)$ is the limit graph of the subsequence consisting only of the stage graphs. 
Thus, in both cases the limit graph $G(S^*)$ is unique up to isomorphism.
\end{proof}

\begin{rem}
We do not know whether the uniqueness holds for non-connected two-rooted graphs. The difficulty is to extend~\cref{lemma:stage-ext}.
In particular, the analogue of~\cref{stage-j-graphs-and-extensions}, that the stage $j$ graph is the $(\hat H,\hat s,\hat t)$-extension of the stage $j-1$ graph is not correct anymore. For example, let $(\hat H, \hat s, \hat t)$ be the two-rooted graph such that $\hat H$ is the $4$-vertex graph consisting of two isolated edges $\{\hat s, u\}$ and $\{\hat t, v\}$.
In this case, the stage $1$ graph is the $(\hat H,\hat s,\hat t)$-extension of the stage $0$ graph, and consists of two disjoint copies of $P_4$. 
The stage $2$ graph is also the $(\hat H,\hat s,\hat t)$-extension of the stage $1$ graph, and consists of two disjoint copies of $P_6$. 
However, the stage $3$ graph consists of two copies of the graph obtained from $P_6$ by adding a pendant edge to every vertex. 
This graph differs from the $(\hat H,\hat s,\hat t)$-extension of the stage $2$ graph, which consists of two disjoint copies of $P_8$.
\end{rem}

By \cref{thm:unique-proper-extension}, every connected two-rooted graph $(\hat H, \hat s, \hat t)$ admits a proper PE-sequence, and limit graphs of all such sequences are isomorphic to the same graph.
We will call it the \emph{limit graph} of $(\hat H, \hat s, \hat t)$ and denote it by $G(\hat H, \hat s, \hat t)$.

\begin{cor}\label{cor:infiniteG*}
No connected two-rooted graph can have both a finite and an infinite PE-sequence.
\end{cor}

\begin{proof}
Suppose for a contradiction that a two-rooted graph $(\hat H, \hat s, \hat t)$ has both a finite PE-sequence $S$, as well as an infinite one, say $S'$.
Since $S$ is finite, it is proper.
By~\cref{thm:unique-proper-extension}, the limit graph of $S'$ is a subgraph of the limit graph of some proper PE-sequence $S''$ of $(\hat H, \hat s, \hat t)$.
Since $\hat H$ is connected, the same theorem also implies that the limit graphs of $S$ and $S''$ are isomorphic.
Hence, the infinite graph $G(S')$ is a subgraph of the limit graph of $S$, which is finite, a contradiction.
\end{proof}

\cref{cor:infiniteG*} implies the following.

\begin{cor}\label{infinite PE-sequences}
Let $(\hat H, \hat s, \hat t)$ be a connected two-rooted graph that has an infinite PE-sequence.
Then $(\hat H, \hat s, \hat t)$ is PE-inherent.
\end{cor}

\cref{TreelyNonInherent} and~\cref{cor:infiniteG*} also imply the following statement.

\begin{observation}\label{cor:finiteG*}
A connected two-rooted graph $(\hat H, \hat s, \hat t)$ is PE-inherent if and only if its limit graph is infinite.
Hence, $(\hat H, \hat s, \hat t)$ is non-inherent if its limit graph is finite.
\end{observation}

Informally, PE-sequences provide necessary steps to confine a two-rooted graph $(\hat H, \hat s,\hat t)$. Nevertheless, it is possible that some confining graph for  $(\hat H, \hat s,\hat t)$ does not contain an induced subgraph isomorphic to the limit graph of $(\hat H, \hat s,\hat t)$; see \cref{pic:example-confining}. 
We suggest the following weaker conjecture.

\begin{conjecture}
Any graph that confines $(\hat H, \hat s,\hat t)$ contains a not necessarily induced subgraph isomorphic to the limit graph of $(\hat H, \hat s,\hat t)$ provided that the latter is finite. 
\end{conjecture}

\begin{figure}[!h]
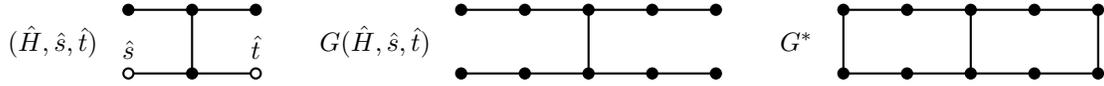

  \centering
  \mpfile{ext}{70}\qquad
  \mpfile{ext}{71}\qquad
  \mpfile{ext}{72}

  \caption{A two-rooted graph $(\hat H, \hat s,\hat t)$, its limit graph $G(\hat H, \hat s,\hat t)$, and a confining graph $G^*$ of $(\hat H, \hat s,\hat t)$ that does not contain any induced subgraph isomorphic to $G(\hat H, \hat s,\hat t)$.}\label{pic:example-confining}
\end{figure}

We prove this conjecture for the case of trees.
Note that in this case, the subgraph and the induced subgraph relations coincide.

\begin{prop}
  \label{prop:PE-confining}
  Let $(\hat H, \hat s,\hat t)$ be a connected two-rooted graph such that $\hat H$ is a tree and the limit graph $G(\hat H, \hat s,\hat t)$ is finite. 
  Then any confining tree of $(\hat H, \hat s,\hat t)$ contains a subgraph isomorphic to $G(\hat H, \hat s, \hat t)$.
 \end{prop}

\begin{proof}
Towards a contradiction let $G^*$ be a confining tree of $(\hat H, \hat s, \hat t)$ that does not contain $G(\hat H, \hat s, \hat t)$ as a subgraph.
Fix an arbitrary proper PE-sequence $S = (G_0,G_1,\ldots,G_k)$ of $(\hat H, \hat s, \hat t)$. 
Then $G(\hat H, \hat s, \hat t)$ is isomorphic to $G_k$.

Note that $\hat H = G_0$ is a subgraph of $G^*$.
Let $i\in \{0,1,\dots,k\}$ be the maximum integer such that $G_i$ is isomorphic to a subgraph of $G^*$.
By our assumptions we have $i<k$.
Furthermore, let $(H,s,t)$ be a simplicial copy of $(\hat H, \hat s, \hat t)$ in $G_i$ such that $G_{i+1}$ is a PE of $G_i$ with respect to $(H,s,t)$.
Since $H$ is a subgraph of $G_i$ and $G_i$ is isomorphic to a subgraph of $G^*$, we obtain that $H$ is also isomorphic to a subgraph of $G^*$.
Note that the copy $(H,s,t)$ of $(\hat H, \hat s, \hat t)$ in $G^*$ is not simplicial, as that would contradict the fact that $G^*$ confines $(\hat H, \hat s, \hat t)$.
Fix an extension $(H',s',t')$ of $(H,s,t)$ in $G^*$ such that $H'$ is a subgraph of $G^*$ obtained from $H$ by adding to it two pendant edges $ss'$ and $tt'$. 
By the definition of $G_{i+1}$, a graph isomorphic to $G_{i+1}$ can be obtained from $G_i$ by adding to it at least one of the pendant edges $ss'$ and $tt'$.
Since $G^*$ is a tree and $G_i$ is connected, neither $s'$ nor $t'$ belongs to a copy of $G_i$, since otherwise a cycle in $G^*$ would appear. 
Thus we conclude that $G_{i+1}$ is isomorphic to a subgraph of $G^*$. 
But this contradicts the maximality of $i$.
\end{proof}

\subsection{Proof of \cref{th:tree-deg-conditions}}

Recall that every inherent two-rooted graph is PE-inherent (\Cref{inherent-implies-PE-inherent}).
Thus, in order to prove~\Cref{th:tree-deg-conditions}, it suffices to prove the following.

\begin{lem}\label{necessary condition for PE-inherent}
Let $(\hat H, \hat s, \hat t)$ be a PE-inherent connected two-rooted graph such that $d_{\hat H}(\hat s)\leq d_{\hat H}(\hat t)$. 
Then $\hat H$ is a subcubic two-rooted tree.
\end{lem}

Before proving~\cref{necessary condition for PE-inherent}, let us point out that the given condition is only a necessary condition for PE-inherence, but not a sufficient one.

\begin{example}\label{exampleF131}
  \Cref{pic:F131ext} presents an example of  a  two-rooted subcubic
  tree $T$ that is not PE-inherent. 
  The  stage 1 graph of $T$ contains no simplicial
    copies of $T$.
\begin{figure}[!h]
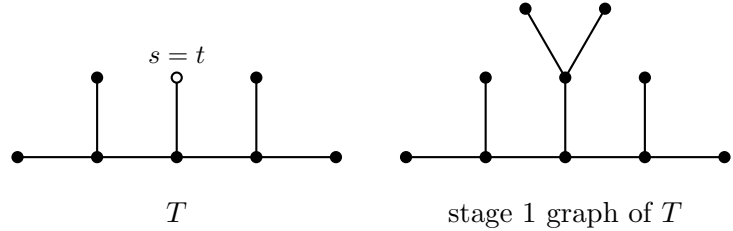

\centering
\begin{tabular}{c@{\qquad}c}
  \mpfile{ext}{60}&\mpfile{ext}{61}\\[3mm]
  $T$ & stage 1 graph of $T$
\end{tabular}
  \caption{An example of a two-rooted subcubic tree $T$ that is not PE-inherent}\label{pic:F131ext}
\end{figure}
\end{example}

In fact, it does not seem to be easy to characterize PE-inherent subcubic two-rooted trees (see~\cref{sec:PE-inherent} for more details).

\begin{proof}[Proof of~\Cref{necessary condition for PE-inherent}.]
Let $(\hat H,\hat s,\hat t)$ be a connected PE-inherent two-rooted graph and let $S=(G_0,G_1,\ldots)$ be an arbitrary but fixed stage sequence of $(\hat H,\hat s,\hat t)$. 
Then $S$ is infinite.
Let $G$ be the limit graph of $S$ (or, equivalently, the limit graph of $(\hat H,\hat s,\hat t)$).  

Suppose that $d_{\hat H}(\hat s)\geq 2$.
Then every vertex of $G$ is either in $G_0$ or has a neighbor in $G_0$. 
Indeed, to obtain $G$ one should add pendant edges to $G_0 = \hat H$. 
But pendant vertices of these edges cannot be mapped to either $\hat s$ or $\hat t$ in a copy of $(\hat H,\hat s,\hat t)$ since the degree of pendant vertices is~$1$ in the extended graph.
Thus, the stage $1$ graph of $S$ is the last element of $S$.
This implies that $S$ is finite, a contradiction.
Hence $d_{\hat H}(\hat s)\leq 1$.

If $d_{\hat H}(\hat s)= 0$, then $\hat H$ is a path (of length zero) where endpoints $\hat s$ and $\hat t$ coincide.
From now on, we assume that $d_{\hat H}(\hat s)= 1$.

\begin{claim}\label{limit-graph-has-bounded-degree}
For every vertex $v$ of the limit graph $G$, we have $d_G(v)\le \Delta(\hat H)+2$.
\end{claim}

\begin{proof}
Suppose by contradiction that $d_G(v)>\Delta(\hat H)+2$ for some $v\in V(G)$, and let $i$ be the smallest integer such that $v$ is a vertex of the stage $i$ graph of $S$ and the degree of $v$ in this graph is more than $\Delta(\hat H)+2$.
Then $i\ge 1$ and $v$ is a vertex of the stage $i-1$ graph of $S$ (since otherwise it would have degree one in the stage $i$ graph of $S$).
By \cref{stage-j-graphs-and-extensions}, the stage $i$ graph of $S$ is the \hbox{$(\hat H,\hat s,\hat t)$-extension} of the stage $i-1$ graph of $S$.
From \cref{Hst-extension} it follows that the degree of $v$ in the stage $i$ graph of $S$ is at most $\max\{d_{\hat H}(\hat s),d_{\hat H}(\hat t)\}+2\le \Delta(\hat H)+2$, a contradiction.
\end{proof}

We next show that $\hat H$ is a tree. 
We will make use of a rooted forest $F$ called the \emph{pendant forest} and defined as the graph with vertex set $V(G)$ and edge set $E(G)\setminus E(\hat H)$, whose components are trees rooted at vertices in $H$.
Suppose for a contradiction that $\hat H$ contains a cycle.
Then the set $\bar B(\hat H)$ of non-bridges of $\hat H$ is non-empty. 
Due to the No New Cycle Property, $\bar B(H) = \bar B(\hat H)$ for any copy $H$ of $\hat H$ in an extended graph. 
For $v\in V(\hat H)$, let $\ell_v$ be the distance in $\hat H$ from $v$ to the set of vertices incident with edges of $\bar B(\hat H)$.
Then the depth of $F$ is at most $\ell = \max\{\ell_{\hat s},\ell_{\hat t}\}$. 
Indeed, any vertex at larger distance from the initial two-rooted graph cannot be mapped to either $\hat s$ or $\hat t$ in a copy of $(\hat H,\hat s,\hat t)$ in an extended graph.
Together with the fact that the limit graph $G$ has vertex degrees bounded by $\Delta(\hat H)+2$ by~\cref{limit-graph-has-bounded-degree}, this implies that $G$ is finite, a contradiction with \cref{cor:finiteG*}.
Thus, $\hat H$ is a tree, as claimed.

By \cref{limit-graph-has-bounded-degree} we get the following.

\begin{claim}\label{cl:intersectG0}
There exists an integer $i_0$ such that for all $i \ge i_0$ no simplicial copy of $(\hat H, \hat s, \hat t)$ in $G_i$ intersects $G_0$. 
\end{claim}

\begin{proof}
Using \cref{limit-graph-has-bounded-degree} we conclude that there are finitely many copies of $\hat H$ in the limit graph $G$ that intersect $G_0$.
Let us denote by $W$ the set of all vertices contained in some copy of $\hat H$ in $G$ that intersects $G_0$.
This set $W$ is a subset of the vertex set of the stage $j$ graph of $S$, for some $j$.
Then, starting from the stage $j+1$ graph of $S$, the assertion of the observation is true.
\end{proof}

Suppose for a contradiction that $d_{\hat H}(t)\geq 3$. 
We show that the pendant forest $F$ is a disjoint union of paths.
Let $v\in V(G)\setminus V(G_0)$ be an arbitrary vertex of a pendant forest in $G$. 
Clearly, at the time $v$ is added, its degree is one.
Furthermore, if $d_G(v)\neq 1$, then there exists a minimal integer $i$ such that $d_{G_i}(v)=2$, and $v$ is the $s$-vertex in some copy of $(\hat H,\hat s,\hat t)$ in $G_{i-1}$. 
Now observe that, since $d_{\hat H}(t)\ge 3$, the vertex $v$ is not the $t$-vertex of any copy of $(\hat H,\hat s,\hat t)$ in subsequent graphs $G_j$ for $j>i$, which implies $d_G(v)=2$.
This shows that $F$ is a disjoint union of paths, as claimed.
In particular, this implies that in every copy of $(\hat H,\hat s,\hat t)$ in each graph from $S$ the $t$-vertex belongs to $G_0$, which contradicts \cref{cl:intersectG0}.
Thus, $d_{\hat H}(t)\le 2$ is proven.

Note that for any $i>0$, we have $d_{G_i}(v)\le 3$ for any $v\in V(G_i)\setminus V(G_0)$.
In other words, all vertices in $G_i$ of degree at least $4$ belong to $G_0$. Together with \cref{cl:intersectG0} this implies that $\hat H$ does not admit any vertex of degree at least $4$.

Finally, assume that $d_{\hat H}(\hat t)=1$, $\hat s\ne\hat t$, and $\hat H$ is not a path. 
Then, for any $i>0$, the degree of any vertex $v\in V(G_i)\setminus V(G_0)$ cannot exceed $2$.
Since $\hat H$ is not a path, it has a vertex of degree $3$, and any vertex of degree $3$ in $G_i$ belongs to $G_0$, in contradiction with  \cref{cl:intersectG0}.
This completes the proof of~\cref{necessary condition for PE-inherent}.
\end{proof}

\subsection{Automorphisms and preserving PE-inherence}

Consider a PE-inherent two-rooted graph $(\hat H,\hat s_1,\hat t_1)$.
For which pairs of roots $\hat s_2$, $\hat t_2\in V(H)$ is the corresponding two-rooted graph $(\hat H,\hat s_2,\hat  t_2)$ also PE-inherent?
In this subsection we provide a sufficient condition, which will be used in \cref{sec:PE-inherent} to give examples of PE-inherent two-rooted graphs.

Fix a graph $\hat H$.
For a vertex $v\in V(\hat H)$, we define the \emph{orbit} of $v$, denoted by $\orb v$, as the set of vertices $w\in V(\hat H)$ such that there exists an automorphism of $\hat H$ mapping $v$ to $w$. 
We call a pair of two-rooted graphs $(\hat H,\hat s_1,\hat t_1)$ and $(\hat H,\hat s_2,\hat t_2)$ \emph{equivalent} (to each other) if $\hat s_2\in \orb{\hat s_1}$ and $\hat t_2\in \orb{\hat t_1}$. 
Using this notation one can easily obtain the following claim.

\begin{observation}\label{obs:stage1}
Let $\hat s$ and $\hat t$ be two distinct vertices of $\hat H$ and let $S$ be any stage sequence of $(\hat H,\hat s,\hat t)$. 
Then the stage $1$ graph of $S$ is the graph obtained from $\hat H$ by adding one pendant edge to each vertex in $\orb{\hat s}\cup \orb{\hat t}$.\qed
\end{observation}

This observation is illustrated in \cref{pic:observation}, for the extended claw graph, which is the graph depicted in \cref{pic:extended-claw}.

\begin{figure}[!h]
  \centering
  \begin{minipage}[t]{0.45\textwidth}
    \centering
\raisebox{17pt}{\mpfile{more-graphs}{2}}
    \caption{The extended claw graph. 
The orbits of $s$ and $t$ are the sets of vertices of degree $1$ and $2$, respectively.}\label{pic:extended-claw}
  \end{minipage}
  \qquad
  \begin{minipage}[t]{0.45\textwidth}
    \centering
  \mpfile{ext}{52}
  \caption{The stage $1$ graph of the extended claw. Each vertex is labeled by the minimal $i$ such that the stage $i$ graph contains this vertex.
}\label{pic:observation}
  \end{minipage}
\end{figure}


\begin{prop}\label{orbit-lemma}
Let $(\hat H,\hat s_1,\hat t_1)$ and $(\hat H,\hat s_2,\hat t_2)$ be two equivalent two-rooted graphs, and let $G$ be a graph.
Then the $(\hat H,\hat s_1,\hat t_1)$- and $(\hat H,\hat s_2,\hat t_2)$-extensions of $G$ are isomorphic. 
\end{prop}

Although the claim seems intuitively clear due to the fact that  $(\hat H,\hat s_1,\hat t_1)$ and $(\hat H,\hat s_2,\hat t_2)$ are equivalent, we prefer to give a formal proof.

\begin{proof}
If $\hat s_1=\hat t_1$, then the two-rooted graphs $(\hat H,\hat s_1,\hat t_1)$ and $(\hat H,\hat s_2,\hat t_2)$ are isomorphic. 
This implies that the $(\hat H,\hat s_1,\hat t_1)$- and $(\hat H,\hat s_2,\hat t_2)$-extensions of $G$ are isomorphic. 

Assume now that $\hat s_1\neq \hat t_1$. 
By \cref{Hst-extension} it is enough to show that for all vertices $v\in V(G)$, the following conditions are equivalent:
\begin{itemize}
    \item there exists a simplicial copy $(H,s_1,t_1)$ of $(\hat H,\hat s_1,\hat t_1)$ in $G$ such that $v\in\{s_1,t_1\}$ and $d_{G}(v) = d_{H}(v)$, 
    \item there exists a simplicial copy $(H,s_2,t_2)$  of $(\hat H,\hat s_2,\hat t_2)$ in $G$ such that  $v\in \{s_2,t_2\}$ and $d_{G}(v) = d_{H}(v)$.
\end{itemize}
Indeed, this implies that the sets of vertices of $G$ to which pendant edges are added to obtain the $(\hat H,\hat s_1,\hat t_1)$- and $(\hat H,\hat s_2,\hat t_2)$-extensions of $G$, respectively, are the same.
 
Suppose that the condition of the first item is satisfied.
We present the proof for the case $v=s_1$; the arguments for the case $v = t_1$ are the same.
Let $\hat \tau$ be an automorphism of $\hat H$ mapping $\hat s_2$ to $\hat s_1$. 
(Such an automorphism exists due to the fact that $(\hat H,\hat s_1,\hat t_1)$ and $(\hat H,\hat s_2,\hat t_2)$ are equivalent.)
Fix an arbitrary isomorphism $\psi$ from $\hat H$ to $H$ such that $\psi(\hat s_1) = s_1$ and $\psi(\hat t_1) = t_1$.
Then $\varphi= \psi\circ\hat \tau$ is an isomorphism from $\hat H$ to $H$.
We construct the desired two-rooted graph $(H,s_2,t_2)$ by setting 
\begin{align*}
    s_2&=\varphi(\hat s_2)\\
    t_2&=\varphi(\hat t_2).
\end{align*}
Note that $s_2 = \varphi(\hat s_2) =\psi(\hat \tau(\hat s_2)) = \psi(\hat s_1) = s_1 = v$.
Since $s_2 = s_1= v$ and $d_{G}(v) = d_{H}(v)$, we have $d_{G}(s_2) = d_{H}(s_2)$.
In particular, this implies that $(H,s_2,t_2)$ is a simplicial copy of $(\hat H,\hat s_2,\hat t_2)$ in $G$.

The proof of the other direction is similar.
\end{proof}

The following statement also holds.

\begin{thm}\label{lem:isomorphic-limits-of-equivalent-trgs}
Let $(\hat H,\hat s_1,\hat t_1)$ and $(\hat H,\hat s_2,\hat t_2)$ be two connected equivalent two-rooted graphs. 
Then their limit graphs are isomorphic. 
In particular, $(\hat H,\hat s_1,\hat t_1)$ is PE-inherent if and only if $(\hat H,\hat s_2,\hat t_2)$ is.
\end{thm}

\begin{proof}
For $i\in \{1,2\}$, let $S_i$ be a stage sequence of $(\hat H,\hat s_i,\hat t_i)$.
By~\cref{stage-j-graphs-and-extensions}, for any $j\ge 1$, the stage $j$ graph of $S_i$ coincides with the \hbox{$(\hat H,\hat s_i,\hat t_i)$-extension} of the stage $j-1$ graph of $S_i$.
Since the stage $0$ graph is in both cases $\hat H$, an induction on $j$ along with~\cref{orbit-lemma} implies that for any $j\ge 1$, the stage $j$ graphs of $S_1$ and $S_2$ are isomorphic.
Thus, the limit graphs of $(\hat H,\hat s_1,\hat t_1)$ and $(\hat H,\hat s_2,\hat t_2)$ are isomorphic. 

Finally, we show that $(\hat H,\hat s_1,\hat t_1)$ is PE-inherent if and only if $(\hat H,\hat s_2,\hat t_2)$ is.
By symmetry, it suffices to show that if $(\hat H,\hat s_1,\hat t_1)$ is PE-inherent, then so is $(\hat H,\hat s_2,\hat t_2)$.
Assume that $(\hat H,\hat s_1,\hat t_1)$ is PE-inherent. Then the limit graph of $(\hat H,\hat s_1,\hat t_1)$ is infinite.
Hence, so is the limit graph of $(\hat H,\hat s_2,\hat t_2)$, which means the sequence $S_2$ is infinite.
By~\cref{cor:infiniteG*}, all PE-sequences of $(\hat H,\hat s_2,\hat t_2)$ are infinite.
Thus, due to \cref{infinite PE-sequences}, $(\hat H,\hat s_2,\hat t_2)$ is PE-inherent.
\end{proof}

\subsection{On the PE-inherent graphs}\label{sec:PE-inherent}

\Cref{necessary condition for PE-inherent} gives  necessary conditions for a connected two-rooted graph $H$ to be PE-inherent.
Although these conditions 
are rather strong,
there are many PE-inherent two-rooted graphs.
It seems difficult to characterize them; however, we provide six infinite families of examples.
Three of them consist of two-rooted combs, which are defined as follows.

\begin{defn}
For integers $p,q,r\ge 0$ with $p+q+r>0$ we denote by $F(p,q,r)$ the graph consisting of a path $P_{p+q+r}=(a_1,\dots,a_{p+q+r})$ and $q$ pendant edges added to vertices ${a_{p+1},\dots,a_{p+q}}$, with the other endpoints ${b_{p+1},\dots,b_{p+q}}$, respectively (see \cref{pic:F211} for an example); in particular, if $q = 0$, then no pendant edges are added.
Any graph of this type will be referred to as a \emph{comb}.
Furthermore, any subcubic two-rooted tree $(H,s,t)$ such that $H$ is a comb will be referred to as a \emph{two-rooted comb}.
\end{defn}

Special cases of two-rooted combs can be obtained when the underlying graph is a path.
An \emph{endpoint-rooted path} is a two-rooted graph $(\hat H, \hat s, \hat t)$ such that $\hat H$ is a path and $\hat s$ and $\hat t$ are its endpoints.
A \emph{one-endpoint-rooted path} is a two-rooted graph $(\hat H,\hat s,\hat t)$ such that $\hat H$ is a path and at least one of $\hat s$ and $\hat t$ is an endpoint of this path.
Recall that we always assume $d_{\hat H}(\hat s)\le d_{\hat H}(\hat t)$, hence $\hat s$ is always an endpoint of the path.

Obviously, every endpoint-rooted path is also a one-endpoint-rooted path.
Note that every one-endpoint-rooted path is a two-rooted comb $(\hat H,\hat s,\hat t)$ where $\hat H = F(\ell,0,\ell')$ for some $\ell$ and~$\ell'$.

\medskip

\noindent\textbf{Two-rooted combs of type I:} For integers $p\ge 1$ and $q,r\ge 0$, we denote by $T_1(p,q,r)$ the two-rooted graph $(F(p,q,r),a_1,a_p)$.
Any such two-rooted graph will be referred to as a \emph{two-rooted comb of type I}.
See \Cref{figure of comb of type I} for an example.

\begin{figure}[!h]
\centering
\mpfile{more-graphs}{211}
\caption{$F(2,1,1)$, also known as the fork graph.\label{pic:F211}}
\vspace{5mm}
\centerline{\mpfile{pi2rg}{1}}
  \caption{$T_1(4,3,3)$, a two-rooted comb of type I.\label{figure of comb of type I}}
\end{figure}

\begin{prop}\label{combs-of-type-I-are-PE-inherent}
  All two-rooted combs of type I are PE-inherent.
\end{prop}
\begin{proof}
  Let $p\ge 1$ and $q,r\ge 0$ and consider the corresponding two-rooted comb of type I.
  By~\cref{infinite PE-sequences}, it is enough to provide an infinite PE-sequence of $T_1(p,q,r)$. 
   If $q=r=0$ then $T_1(p,q,r)$ is isomorphic to an endpoint-rooted path, which is PE-inherent due to \cref{th:endpoint-rooted paths are inherent,inherent-implies-PE-inherent}.
  So assume $q+r> 0$ and observe that the sequence of graphs $G_i = F(p,q+i,r)$, $i\ge 0$, is an infinite PE-sequence of $T_1(p,q,r)$. 
\end{proof}

Since every one-endpoint-rooted path is isomorphic to a two-rooted comb of type I, namely, $T_1(p,0,r)$ for some $p$ and $r$, the above proposition implies the following.

\begin{cor}
\label{one-endpoint-rooted paths are PE-inherent}
All one-endpoint-rooted paths are PE-inherent.
\end{cor}

\noindent\textbf{Two-rooted combs of type II:}
For integers $p,q\ge 1$ we denote by $T_2(p,q)$ the two-rooted graph $(F(p,q,p),a_1,a_{p+q+1})$.
Any such two-rooted graph will be referred to as a \emph{two-rooted comb of type II}.
See \Cref{figure of comb of type II} for an example.

\begin{figure}[!h]
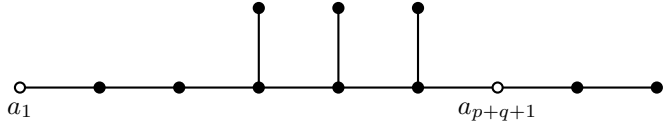

  \centerline{\mpfile{pi2rg}{2}}
  \caption{$T_2(3,3)$, a two-rooted comb of type II}\label{figure of comb of type II}
\end{figure}
    
\cref{lem:isomorphic-limits-of-equivalent-trgs,combs-of-type-I-are-PE-inherent} imply the following.

\begin{prop}\label{combs-of-type-II-are-PE-inherent}
  All two-rooted combs of type II are PE-inherent.
\end{prop}

\begin{proof}
Let $p,q\ge 1$ and consider the corresponding two-rooted comb of type II.
Note that 
$a_p\in \orb{a_{p+q+1}}$
in the underlying graph $F(p,q,p)$.
Therefore, the two-rooted graphs $T_1(p,q,p) = (F(p,q,p),a_1,a_{p})$ and $T_2(p,q) = (F(p,q,p),a_1,a_{p+q+1})$ are equivalent.
By \cref{lem:isomorphic-limits-of-equivalent-trgs,combs-of-type-I-are-PE-inherent}, the two-rooted graph $T_2(p,q)$ is PE-inherent.
\end{proof}

\noindent\textbf{Two-rooted combs of type  III:} 
For integers $p\ge 1$ and $q\ge 0$ we denote by $T_3(p,q)$ the two-rooted graph $(F(p,q,p+1),a_1,a_{p+q+1})$.
Any such two-rooted graph will be referred to as a \emph{two-rooted comb of type III}.
\begin{prop}\label{combs-of-type-III-are-PE-inherent}
  All two-rooted combs of type III are PE-inherent.
\end{prop}

\begin{proof}
Let $p\ge 1$ and $q\ge 0$ and consider the corresponding two-rooted comb $T_3(p,q)$. 
We provide an infinite PE-sequence of $T_3(p,q)$ by setting 
\[
\begin{aligned}
    G_0&=F(p,q,p+1)\,, &\\
    G_{2i-1}&=F(p+1,q+i,p) & \text{for all $i\ge 1$\,,}\\
    G_{2i}&=F(p+1,q+i,p+1) & \text{for all $i\ge 1$\,.}
\end{aligned}
\]
Indeed:
\begin{itemize}
\item we obtain a graph isomorphic to $G_1$ from $G_0$ by adding a pendant edge to both roots in the unique copy of $T_3(p,q)$ in $G_0$;
\item for all $i\ge 1$, we obtain $G_{2i}$ from $G_{2i-
1}$ by adding a new vertex $a_{2p+q+i+2}$ and making it adjacent to the $s$-vertex of a particular simplicial copy of $T_3(p,q)$ in $G_{2i-1}$;
\item for all $i\ge 1$, we obtain a graph isomorphic to $G_{2i+1}$ from $G_{2i}$ by adding a pendant edge to the $t$-vertex of a particular simplicial copy of $T_3(p,q)$ in $G_{2i}$.\qedhere
\end{itemize}
\end{proof}

As shown above,
all two-rooted combs of types I, II, or III are PE-inherent.
It turns out that these are the only PE-inherent two-rooted combs. To verify this, one can use 
the following exhaustive list of conditions that classify all two-rooted combs $(\hat H,\hat s,\hat  t)$ up to isomorphism:
\begin{enumerate}
\item $\hat H$ is a one-endpoint-rooted path (in which case $\hat H$ is PE-inherent by \cref{one-endpoint-rooted paths are PE-inherent}).
\item $\hat H=F(p,q,r)$ with positive $p,q,r$, and the two roots $\hat s$ and $\hat t$ satisfy at least one of the following conditions:
\begin{enumerate}[a)]
\item $\hat s = a_1$ and $\hat t = a_i$ for some $i\in\{1,\ldots, p\}$ (in which case $H$ is PE-inherent if $i = p$ by \cref{combs-of-type-I-are-PE-inherent}),
\item $\hat s = a_1$ and $\hat t = a_{p+q+i}$ for some $i\in\{1,\ldots, r-1\}$ (in which case $\hat H$ is PE-inherent if $i=1$ and $r\in \{p,p+1\}$ by \cref{combs-of-type-II-are-PE-inherent,combs-of-type-III-are-PE-inherent}),
\item $\hat s =\hat  t = b_{p+i}$ for some $i\in \{1,\ldots, q\}$  (in which case $H$ is PE-inherent if $p=i=1$ or $(r,i) = (1,q)$ by \cref{combs-of-type-I-are-PE-inherent}), 
\item $\hat s = b_{p+i}$ for some $i\in \{1,\ldots, q\}$ and $\hat t = a_j$ for some $j\in\{2,\ldots, p\}$.
\end{enumerate}
\end{enumerate}
In fact, for any two-rooted comb which is not of type I, II, or III, the limit graph is isomorphic to the stage $1$ graph (cf.~\cref{exampleF131}).
We leave the details to the reader.

\medskip
Now we provide three more families of PE-inherent two-rooted graphs.
PE-inherence of these families can be established using arguments similar to those used in the proofs of \cref{combs-of-type-II-are-PE-inherent,combs-of-type-III-are-PE-inherent}.
Since these results are not important for the rest of the paper, we again leave the details to the careful reader.

\medskip
\noindent\textbf{Two-rooted leaf-extended full trees.}
For an integer $d\ge 2$, a \emph{full depth-$d$ tree} is a tree with radius $d$ in which every vertex has degree $1$ or $3$, and there are exactly $3^d$ pendant vertices.
A \emph{two-rooted leaf-extended full tree} is any two-rooted graph obtained from the full depth-$d$ tree (for some $d\ge 2$) by extending every leaf with a pendant edge and choosing $s$ and $t$ as arbitrary vertices of degree one and two, respectively.
See \Cref{figure of type IX} for an example.

\begin{figure}[!h]
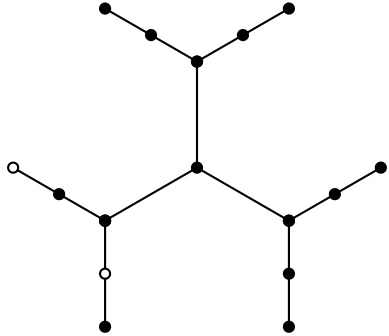

  \centerline{\mpfile{pi2rg}{8}}
  \caption{A two-rooted leaf-extended full tree, $d = 2$}\label{figure of type IX}
\end{figure}

\medskip
\noindent\textbf{Two-rooted rakes.}
For an integer $q\ge 2$ we denote by $T_4(q)$ the two-rooted graph obtained from the two-rooted comb $T_2(2,q)$ of type II by subdividing each edge of the form $a_ib_i$, $i = 3,\ldots, q+2$. 
Any two-rooted graph equivalent to $T_4(q)$, for some $q$, will be referred to as a \emph{two-rooted rake}.\footnote{Note that the construction works for  $q=1$, but in this case we get a two-rooted leaf-extended full tree of depth one.}
See \Cref{figure of long comb} for an example.

\begin{figure}[!h]
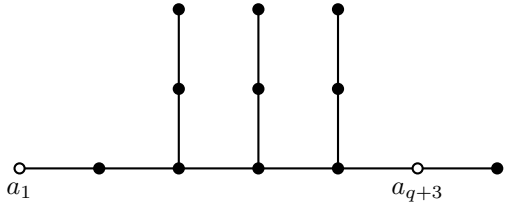

  \centerline{\mpfile{pi2rg}{3}}
  \caption{$T_4(3)$, a two-rooted rake with $3$ teeth}\label{figure of long comb}
\end{figure}

\medskip
\noindent\textbf{Two-rooted split rakes.}
For an integer $q\ge 3$ we denote by $T_5(q)$ the two-rooted graph obtained from the two-rooted comb $T_2(2,q)$ of type II by subdividing each edge of the form $a_ib_i$, $i= 3,\ldots, q+2$, and adding a pendant edge to each vertex of degree two joining $a_i$ and $b_i$ for all $i \in \{4,\ldots, q+1\}$. 
Any two-rooted graph equivalent to $T_5(q)$, for some $q$, will be referred to as a \emph{two-rooted split rake}.
See \Cref{figure of type VIII} for an example.

\begin{figure}[!h]
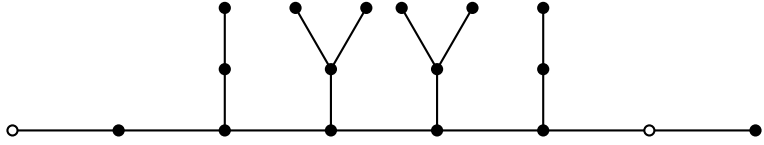

  \centerline{\mpfile{pi2rg}{7}}
  \caption{$T_5(4)$, a two-rooted split rake}\label{figure of type VIII}
\end{figure}

\noindent\textbf{Other examples of PE-inherent two-rooted graphs.}
In \cref{figure of type VIIa} we give two further examples of PE-inherent two-rooted graphs, which we call $T_6$ and $T_7$.
Note that further examples of PE-inherent two-rooted graphs can be obtained using \cref{lem:isomorphic-limits-of-equivalent-trgs}, by considering two-rooted graphs equivalent to $T_6$ or $T_7$.

\begin{figure}[!h]
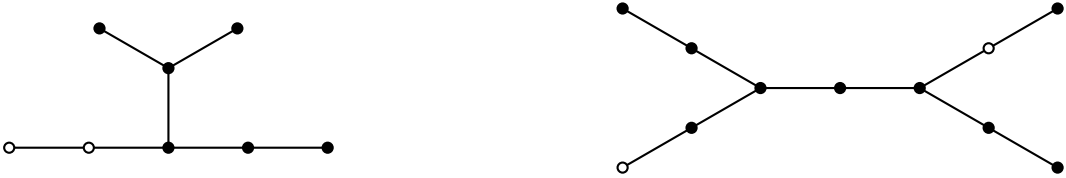

  \begin{minipage}{0.4\textwidth}
  \centerline{\mpfile{pi2rg}{5}}
  \end{minipage}\hskip1.5cm
  \begin{minipage}{0.5\textwidth}
  \centerline{\mpfile{pi2rg}{4}}
 \end{minipage}
 \caption{The two-rooted graphs $T_6$ (left) and $T_7$ (right)}\label{figure of type VIIa}    
\end{figure}

\begin{rem}\label{rem:leftover}
In a PE-inherent two-rooted graph the vertices of degree $3$ do not necessarily form a subtree (see~\cref{figure of type VIIa}).
\end{rem}

\begin{question}
Which subcubic forests can be realized as subgraphs of PE-inherent two-rooted graphs induced by vertices of degree $3$?
\end{question}

\section{PE-inherent two-rooted graphs that are not inherent}\label{sec:killing-list}

In order to prove that a certain PE-inherent two-rooted graph is not inherent, we will use confining graphs defined in \cref{subsec:definitions}.
Note that any confining graph for a two-rooted graph $(\hat H,\hat s,\hat t)$ is a witness of non-inherence of $(H,s,t)$.
In fact, $(\hat H,\hat s,\hat t)$ is inherent if and only if no graph confines it.

Recall that for two integers $k\ge 2$ and $g\ge 3$, a \emph{$(k,g)$-cage} is a $k$-regular graph that has as few vertices as possible given its girth $g$.
        
\begin{lem}\label{lm:case-i}
  Let $(\hat H,\hat s,\hat t)$ be a subcubic two-rooted tree such that
  \begin{enumerate}[(i)]
      \item $\hat s\neq \hat t$ and $\hat t$ is adjacent to a leaf $\hat \ell$ distinct from $\hat s$, and 
      \item $\hat s$ does not admit a false twin. \label{it:no-twin}
  \end{enumerate}
  Then $(\hat H,\hat s,\hat t)$ is non-inherent.
\end{lem}

\begin{proof}
    Let $G_0$ be a $(3, 2|V(\hat H)|)$-cage.  
    By  definition, graph $G_0$ contains the
    full depth-$|V(\hat H)|$ tree as a subgraph. 
    Since $\hat H$ is subcubic, it is a subgraph of any full depth-$|V(\hat H)|$ tree. Hence, it is a subgraph of  $G_0$.
Let the graph $G = G_0[2K_1]$ be the lexicographic product of $G_0$ and a non-edge $2K_1$. 
Observe that vertices of any copy of $\hat H$ in $G$ correspond to distinct vertices of $G_0$, since otherwise the girth restriction would be violated.
For the same reason, $G_0$ does not admit any false twins.
Thus, any pair of false twins in $G$ corresponds to the same vertex in $G_0$.

We show that  $G$ is a confining graph for $(\hat H,\hat s,\hat t)$, that is, 
$G$ contains no avoidable copy of $(\hat H,\hat s,\hat t)$.
Let $(H,s,t)$ be an arbitrary copy of $(\hat H,\hat s,\hat t)$ in $G$ and let $\ell$ be the vertex of $H$ corresponding to $\hat \ell$.
Since $(\hat H,\hat s,\hat t)$ is a subcubic two-rooted tree with $\hat s\neq \hat t$, the degree of $\hat t$ in $\hat H$ is equal to $2$.
If $\hat \ell$ has a false twin $\hat \ell'$ in $\hat H$, then both of them are neighbors of $\hat t$. Therefore  $\hat H$ is a path $P_3$. 
This implies that  $\hat s = \hat \ell'$, which  contradicts assumption \eqref{it:no-twin}. 
This implies that $\hat \ell$ does not have a false twin in $\hat H$.
Note that two vertices in $H$ can be false twins in $G$ only if they are false twins in $H$.
Since this is not the case for the vertex $s$, and due to the definition of $G$, there exists
a neighbor $s'$ of $s$ in $G$ that does not belong to $H$ such that $s$ is the only neighbor of $s'$ in $V(H)$.
Now define $t'$ as the unique false twin of $\ell$ in $G$, and observe that there exists an extension $(H',s',t')$ of $(H,s,t)$ in $G$.
Clearly this extension cannot be closed without visiting a vertex from $N_G(\ell)$ (see \cref{fig:bad-extension}).

    \begin{figure}[!h]
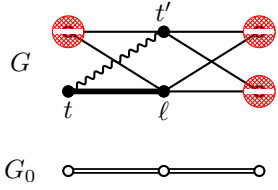

      \centering
        \mpfile{ph}{111} 



      \caption{There is no way to extend $tt'$ to close the extension.}\label{fig:bad-extension}
    \end{figure}

    Since we identified a non-closable extension of an arbitrary copy $(H,s,t)$ of $(\hat H, \hat s, \hat t)$ in $H$, we conclude that the graph $G$ is indeed confining.
    Hence $(\hat H, \hat s, \hat t)$ is non-inherent.
\end{proof}

We construct confining graphs for five infinite families of PE-inherent two-rooted graphs outlined in \cref{sec:PE-inherent}: certain combs of type I, combs of type II, rakes, split rakes, and leaf-extended full trees.

\begin{sloppypar}
\begin{thm}\label{thm:PE-inherent-non-inherent}
The following PE-inherent two-rooted graphs are non-inherent:
\begin{enumerate}
\item Two-rooted combs of type II.\label{it:conf1}
\item Two-rooted rakes and split rakes.\label{it:conf2}
\item Two-rooted leaf-extended full trees with non-adjacent roots.\label{it:conf3}
\item Two-rooted graphs $T_1(\ell-1,0,1)$ for $\ell\geq3$ and $T_1(0,0,\ell)$ for $\ell\geq 1$.

These are two-rooted graphs $(P,s,t)$ such that $P=(v_0,\dots ,v_{\ell})$ is a path of length $\ell\ge 1$, $s = v_0$, and either $t=v_0$ or $(t=v_{\ell-1}$ with $\ell\ge 3)$.
\label{it:conf4}
\item Certain two-rooted combs of type I, including
    $T_1(1,0,2)$, 
    $T_1(1,0,3)$,
    $T_1(1,1,1)$,
    $T_1(1,1,2)$,
    $T_1(1,2,1)$,
    $T_1(1,3,1)$,
    $T_1(1,4,1)$,
    $T_1(2,0,2)$, 
    $T_1(2,0,3)$,
    $T_1(2,1,1)$,
    $T_1(2,1,2)$,
    $T_1(2,2,1)$,    
    $T_1(2,3,1)$,
    $T_1(2,4,1)$, and
    $T_1(3,0,3)$.
\label{it:conf5}
\end{enumerate}
\end{thm}
\end{sloppypar}

To prove \cref{thm:PE-inherent-non-inherent}, we will need other types of confining graphs. 
Given a PE-inherent two-rooted graph $(H,s,t)$, its confining graph may be
\begin{itemize}
    \item a direct modification of the graph $H$, or
    \item an appropriate circulant of small degree, or
    \item an appropriate cage of small degree. 
\end{itemize}
There may be other ways of confining two-rooted graphs. 

We describe these various approaches in the following subsections.
In particular, we prove in \cref{sec:direct-confinement} the non-inherence of two-rooted graphs listed in items \ref{it:conf1}--\ref{it:conf3}
of \cref{thm:PE-inherent-non-inherent}.
The non-inherence of all two-rooted graphs listed in item \ref{it:conf4} is proved in \cref{sec:circulants}, except for the case $s = t= v_0$ and $\ell = 2$, for which  non-inherence is proved  in \cref{sec:cages}, along with the non-inherence of all two-rooted graphs listed in item \ref{it:conf5}.
Note that for a given two-rooted graph, there might be several confining graphs; however, we shall not describe all of them.

\subsection{Direct confinement of some families}\label{sec:direct-confinement}

All combs of type II can be confined by two additional paths of length three (see \Cref{fig:combII-confined}).

\begin{figure}[!h]
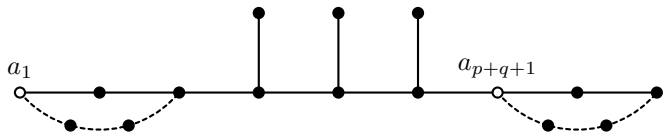

\centerline{\mpfile{pi2rg}{21}}
\caption{Up to automorphism, the graph admits only one copy of $T_2(3,3)$, which is clearly not avoidable, thus showing that $T_2(3,3)$ is confined.}\label{fig:combII-confined}
\end{figure}

Two-rooted rakes and split rakes are confined by similar construction, shown on \cref{fig:rakes-confined}.

\begin{figure}[!h]
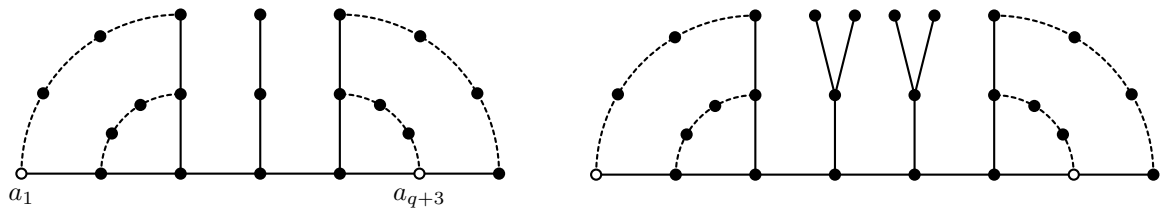

  \centerline{\mpfile{pi2rg}{31} \hskip1cm
\raisebox{10pt}{\mpfile{pi2rg}{72}}}
  \caption{Up to automorphism, the graph on the left admits only one copy of $T_4(3)$, while the graph on the right admits only one copy of $T_5(4)$, both of which are clearly not avoidable, thus showing that these graphs are confined.}\label{fig:rakes-confined}
\end{figure}

The leaf-extended full trees are confined as shown on \cref{fig:fulltrees-confined}.

\begin{figure}[!h]
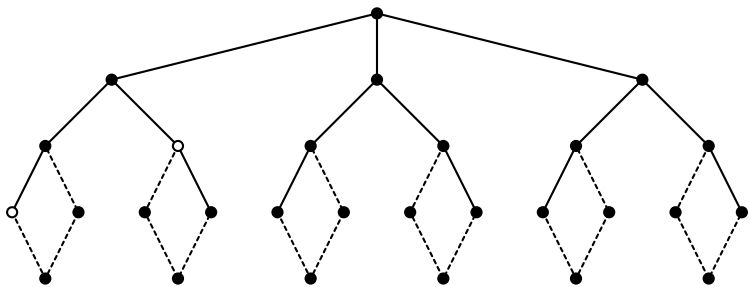

\centerline{\mpfile{pi2rg}{81}}
\caption{Up to automorphism, the graph admits only one copy of the  leaf-extended full tree, which is not avoidable if $s$ and $t$ are not adjacent, thus showing confinement in this case.}\label{fig:fulltrees-confined}
\end{figure}

\subsection{Confinement by circulants}\label{sec:circulants}

In the following lemma we prove non-inherence of two-rooted graphs listed in item \ref{it:conf4} of \cref{thm:PE-inherent-non-inherent}, except for $s = t = v_0$ and $\ell = 2$.

A graph on $n$ vertices is a \emph{circulant} if its vertices can be numbered from $0$ to $n - 1$ in such a way that, if some two vertices numbered $x$ and $(x + d) \pmod n$ are adjacent, then every two vertices numbered $z$ and $(z + d) \pmod n$ are adjacent.
We denote such a graph by $\mathrm{Circ}(n;S)$, where $S$ is the set of all possible values $d$, corresponding to the above definition (see~\cite{zbMATH02149408}).

\begin{lem}
\label{prop:2}
Let $P=(v_0,\dots ,v_{\ell})$ be a path of length $\ell\ge 1$ and let $s$ and $t$ be two vertices of $P$. 
Then, the two-rooted graph $(P,s,t)$ is not inherent if
\begin{enumerate}[(i)]
\item $s=v_0$, $t=v_{\ell-1}$, and $\ell\ge 3$;\label{it:specialcase}
    \item $s=t=v_0$ and $\ell\neq 2$.
\end{enumerate}     
\end{lem}

\begin{proof}
In all cases the confining graph $G$ is the circulant $\mathrm{Circ}(2\ell+6; \{\pm1,\pm(\ell+2)\})$, which is in fact isomorphic to the lexicographic product of cycle $C_{\ell+3}$
and a non-edge $2K_1$.
Note that \eqref{it:specialcase} is a special case of \cref{lm:case-i}.

It is not difficult to check that up to an automorphism, 
 there is a unique path of length $\ell$ in $G$ for every $\ell$, except $\ell=2$; see \Cref{fig:circulants} (a), (b).
 
 \begin{figure}[!h]
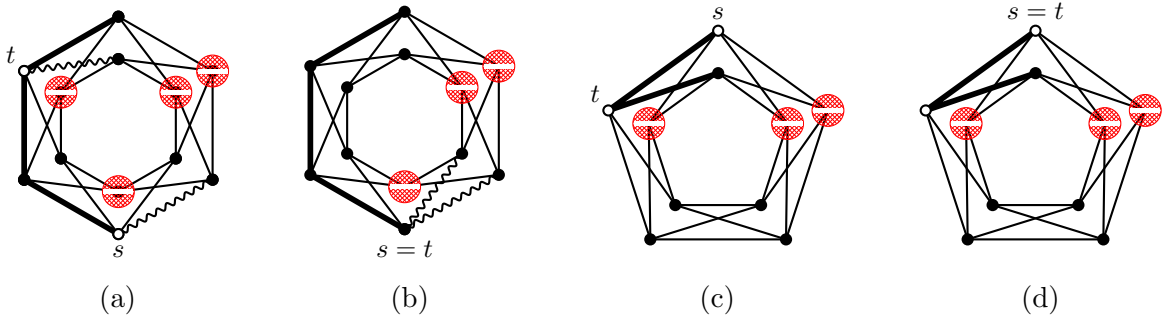

  \centerline{
    \begin{tabular}{c@{\hskip1cm}c@{\hskip1cm}c@{\hskip1cm}c}
    \mpfile{ph}{72} &
    \mpfile{ph}{71} &
    \raisebox{4.5pt}{\mpfile{ph}{74}} &
    \raisebox{4.5pt}{\mpfile{ph}{73}}\\[5pt]
    (a) & (b) & (c) & (d)
    \end{tabular}
  }
  \caption{Circulant graphs $\mathrm{Circ}(2\ell+6; \{\pm1,\pm(\ell+2)\})$ for $\ell= 3$ (in (a) and (b)) and $\ell = 2$ (in (c) and (d)) with induced copies of the corresponding two-rooted graphs $(P,s,t)$, when $t = v_{\ell-1}$ (in (a) and (c)) or $t = s = v_{0}$ (in (b) and (d)).}\label{fig:circulants}
\end{figure}

Furthermore, if $\ell\geq3$ then in both cases, the corresponding two-rooted graph has, up to an automorphism, a  unique extension, which cannot be closed. 
This shows confinement.

For $\ell=1$ both cases give the same two-rooted graph. 
In this case, the graph $G$ is the circulant $\mathrm{Circ}(8; \{\pm1,\pm 3\})$, which is isomorphic to the complete bipartite graph $K_{4,4}$. 
Each induced copy of the corresponding two-rooted graph $(P,s,t)$ in $G$ has, up to symmetry, a unique extension, and this extension cannot be closed.
Again, this shows confinement.
\end{proof}

In case $\ell=2$ the above arguments do not work. There is another path (see \Cref{fig:circulants} (c), (d)) such that, for both cases (i) and (ii), the two-rooted graph $(P,s,t)$ is simplicial, and hence avoidable.

Note that in  case (i) the cycle $C_{\ell+3}$  can be replaced by any longer cycle. In other words, the circulant $\mathrm{Circ}(2k+6; \{\pm1,\pm(k+2)\})$ confines $T_1(\ell-1,0,1)$ for all $k\geq \ell$.

For case (ii) with $\ell=1$ in the proof above, instead of $\mathrm{Circ}(8; \{\pm1,\pm 3\})\cong K_{4,4}$ one may consider also the circulant $\mathrm{Circ}(6; \{\pm1,\pm 3\})$, which is isomorphic to $K_{3,3}$.

Our computations also show that $T_1(2,0,2)$, that is, the two-rooted graph $(P,s,t)$ such that $P=(v_0,v_1,v_2,v_3)$ is a path of length $3$, $s = v_0$, and $t = v_1$, is confined by the circulant $\mathrm{Circ}(20; \{\pm2,\pm5,\pm6\})$ (see \Cref{fig:(2-5-6-20)}). 
\begin{figure}[!h]
  \centerline{\mpfile{circulants}{1} }
  \caption{The circulant $\mathrm{Circ}(20; \{\pm2,\pm5,\pm6\})$}\label{fig:(2-5-6-20)}
\end{figure}

\begin{rem}
If $\ell=0$ then in case (ii) the corresponding two-rooted graph $(K_1,s,t)$ is inherent (see~\cref{th:endpoint-rooted paths are inherent}) and case (i) is impossible.
\end{rem}

\subsection{Confinement by cages}\label{sec:cages}

It turns out that the cages are very useful for confining PE-inherent two-rooted graphs.
In particular, this is true for the Petersen graph, which is a $(3,5)$-cage.

\begin{prop}\label{prop:Petersen}
 The Petersen graph confines $T_1(1,0,2)$, $T_1(2,0,2)$, and $T_1(1,1,1)$.
\end{prop}

\begin{proof}
First notice that the Petersen graph is $3$-arc-transitive (see~\cite[Chapter 27]{lovasz1995handbook} or~\cite[Section 4.4]{MR1829620}).
Similarly, since the closed neighborhood of any vertex induces a claw, the Petersen graph is claw-transitive, i.e., the automorphism group acts transitively on the set of its claws.
Hence it suffices to verify the claim for only one embedding of either $T_1(1,0,2)$, $T_1(2,0,2)$, or $T_1(1,1,1)$ in the Petersen graph (see \cref{fig:T102,fig:T202,fig:T111}).
For each case there is only one extension of the embedding, which is not closable.  
  
\begin{figure}[!h]
  \centerline{
  \begin{minipage}{0.31\textwidth}
  \centering
  \mpfile{ph}{2}
  \caption{$T_1(1,0,2)$\label{fig:T102}}    
  \end{minipage}
  \hfill
  \begin{minipage}{0.31\textwidth}
  \centering
  \mpfile{ph}{1}
  \caption{$T_1(2,0,2)$\label{fig:T202}}    
  \end{minipage}
  \hfill
  \begin{minipage}{0.31\textwidth}
  \centering
  \vspace*{-4.5pt}
  \mpfile{ph}{3}
  \caption{$T_1(1,1,1)$\label{fig:T111}}    
  \end{minipage}
}
\end{figure}
\end{proof}

Further confinements by cages are listed in \cref{tab:cage-confinment}.
The verification for those cases was assisted by computer. 
For more details, including the source code of the verification procedure, we refer the reader to \cite{inherentCode2022}.

\begin{figure}[!h]
  \centering

\raisebox{4pt}{\mpfile{more-graphs}{35}}\quad
  \mpfile{more-graphs}{36}\quad
  \mpfile{more-graphs}{37}\quad
  \mpfile{more-graphs}{38}

    \caption{From left to right: the Petersen, Heawood, McGee, and Tutte--Coxeter graphs}\label{pic:4cages}
  
\end{figure}

\begin{table}[h]
\def\arraystretch{1.5}
    \centering
    \begin{tabular}{c|l}
        Confining graph & Confined two-rooted graphs \\ \hline \hline
        Petersen graph & $T_1(2,0,2)$, $T_1(1,0,2)$, $T_1(1,1,1)$;\\ \hline 
        Heawood graph &  $T_1(1,0,3)$,
             $T_1(1,1,2)$,
             $T_1(1,1,1)$,
             $T_1(1,2,1)$,\\
             &$T_1(2,0,3)$,
             $T_1(2,1,1)$,
             $T_1(2,1,2)$,
             $T_1(3,0,3)$;\\ \hline
        McGee graph & $T_1(2,2,1)$; \\ \hline
        Tutte--Coxeter graph &     
        $T_1(2,3,1)$,
        $T_1(1,3,1)$,
        $T_1(1,4,1)$,
        $T_1(2,4,1)$.\\
        \hline
    \end{tabular}
    \caption{Cages from \Cref{pic:4cages} and some two-rooted graphs confined by them}
    \label{tab:cage-confinment}
\end{table}

We are now ready to prove \cref{th:not-inher}.

\begin{proof}[Proof of \cref{th:not-inher}]
Consider a path $P=(v_0,\dots ,v_{\ell})$ with $s = v_0$ and $t\in V(P)$. 
\begin{itemize}
    \item Assume first that $\ell\ge 1$ and $t=v_0$.
    If $\ell\neq 2$, then the two-rooted path $(P,s,t)$ is not inherent by item 4 of \cref{thm:PE-inherent-non-inherent}.
    The same conclusion holds for the case when $\ell\ge 3$ and $t=v_{\ell-1}$. 
    If $\ell=2$ and $t=v_0$, then $(P,s,t)$ is the two-rooted comb of type I, $T_1(1,0,2)$, which is not inherent by \cref{prop:Petersen}.

\item If $\ell= 3$ and $t=v_{1}$, then $(P,s,t)$ is a two-rooted comb of type I, $T_1(2,0,2)$, which is not inherent by \cref{prop:Petersen}.

\item For the last two cases, if $\ell\in \{4,5\}$ and $t=v_{\ell-3}$, then $(P,s,t)$ is a two-rooted comb of type I, either $T_1(2,0,3)$ or $T_1(3,0,3)$. 
Both are confined by the Heawood graph (see \cref{tab:cage-confinment} and~\cite{inherentCode2022}) and hence not inherent.\qedhere
\end{itemize}
\end{proof}

\subsection{More confining graphs}
\begin{quote}
    \emph{Eight possibly inherent graphs travelling to Devon.\\
    By Dodecahedron one confined and then there were seven.}
\end{quote}
Here we mention additional confining graphs for various PE-inherent two-rooted graphs.
\begin{enumerate}[(i)]
\item Our computations show that the Dodecahedron graph (see \cref{pic:confining-sporadic}) confines $T_1(3,1,1)$.
For this two-rooted graph no other confinements are known.
\item The lexicographic product of the $6$-prism (see \cref{pic:confining-sporadic}) and $2K_1$ confines $T_2(2,1)$.
The 6-prism may be replaced by the 6-M\"obius strip (see \cref{pic:confining-sporadic}).

\item Furthermore, the lexicographic product of  $C_{q'+4}\,\Box\, C_3$ and $2K_1$ confines $T_2(2,q)$, for $0 \le q\le q'$.
\end{enumerate}

\begin{figure}[!h]
  \centering
  \mpfile{more-graphs}{20}\qquad
  \mpfile{more-graphs}{60}\qquad
  \mpfile{more-graphs}{61}
  \caption{From left to right: the Dodecahedron graph, the 6-prism, and the 6-M\"obius strip}
  \label{pic:confining-sporadic}
\end{figure}

\section{Positive results}\label{sec:positive-results}
Up to now, we know few examples of inherent two-rooted graphs. 
The first one is an infinite series consisting of endpoint-rooted paths. 
Their inherence was proved by Bonamy et al.~\cite{MR4245221}. 
Here we provide a more general approach, illustrate it with a proof of the result of Bonamy et al.~\cite{MR4245221} (see \cref{ssec:endpoints}), and use it to prove \cref{th:stv}.

We believe that there are more connected inherent two-rooted graphs. 
Our main candidate is mentioned in \cref{conj-stv}. In our notation it is isomorphic to $T_1(2,0,1) = T_3(1,0)$ (see also \cref{picT110}).
\begin{figure}[h!]
  \centerline{\mpfile{pi2rg}{101}}
  \caption{$T_3(1,0)$}\label{picT110}
\end{figure}

 In \cref{proof-of-th:stv} we provide a partial result supporting this conjecture (\Cref{th:stv}).
Let us mention that all two-rooted combs of type III might be inherent. This is open. In \cref{ssec:more} we list some other open cases. 
In \cref{ssec:disconnected} we give a sufficient condition for inherence of disconnected two-rooted graphs.

\subsection{The inherence of endpoint-rooted paths}\label{ssec:endpoints}

We start with presenting the basic idea of our approach.
Given two graphs $H$ and $G$, we say that a set $U\subseteq V(G)$ is \emph{$H$-avoiding} (or simply avoiding, if $H$ is clear from the context) if the subgraph of $G$ induced by $U$ is connected and the graph $G-N[U]$ contains a copy of $H$.
Now let $\Hst$ be a two-rooted graph and $(H,s,t)$ be its copy in the graph $G-N[U]$, where $U$ is $\hat H$-avoiding.
Then  any extension $(H',s',t')$ with $s',t'$ in $N[U]$ is closable by  a shortest path between $s'$ and $t'$ such that all intermediate vertices belong to $U$. Of course, there are other extensions of the copy. But if one takes an  inclusion-maximal $\hat H$-avoiding set $U$ in $G$, then the analysis of these extensions becomes tractable.
In the case of endpoint-rooted paths (that is, combs of type I of the form $T_1(\ell,0,0)$), this approach gives an alternative proof of the result of Bonamy et al.

\begin{thm}[Bonamy et al.~\cite{MR4245221}]\label{th:st-path-main-short-proof}
Any endpoint-rooted path is inherent.
\end{thm}

\begin{proof}
Let $\hat P$ be a path of length $\ell$ with endpoints $\hat s$, $\hat t$. 
Suppose for a contradiction that the statement fails 
and let $G$ be a minimal counterexample. 
In other words, $G$ contains a copy of $(\hat P,\hat s,\hat t)$ but no avoidable copy 
and this does not happen for any graph having fewer vertices than $G$. 
In particular, no copy of $(\hat P,\hat s,\hat t)$ in $G$ is simplicial. 
Let $(P,s,t)$ be a copy of $(\hat P,\hat s,\hat t)$ in $G$. 
Since the copy is not simplicial, it has an extension, which is a path of length $\ell+2$ with the endpoints $s'$, $t'$. 
Note that $(V(P)\cup \{t'\})\setminus N[s']$ induces a copy of $(\hat P,\hat s,\hat t)$ in $G-N[s']$, and hence the set $\{s'\}$ is $\hat P$-avoiding in $G$.
Fix any inclusion-maximal $\hat P$-avoiding set $U$ in $G$ and let $G' = G-N[U]$.
Since $G'$ contains a copy of $P$, it also contains, by the minimality of $G$, an avoidable copy $(\tilde P, x, y)$ of $(\hat P,\hat s,\hat t)$. 
We come to a contradiction  by showing that $(\tilde P, x, y)$ is avoidable in $G$.
Consider its extension $(\tilde P',x',y')$ in $G$.
If both $x'$ and $y'$ belong to $N(U)$, then the extension can be closed via a path within the connected graph $G[U]$.
If both $x'$ and $y'$ belong to $V(G')$, the extension can be closed via a path within $G'$, since $(\tilde P, x, y)$ is avoidable in this graph.

Suppose now that exactly one of $x'$ and $y'$ belongs to $N(U)$.
Without loss of generality we may assume that $x'\in N(U)$ (and then $y'\in V(G')$).
Set $U' = U\cup\{x'\}$. 
Since $x'$ has a neighbor in $U$, set $U'$ induces a connected subgraph of $G$.
Furthermore, $(V(\tilde P)\cup \{y'\})\setminus N[x']$ induces a copy of $P$ in $G-N[U']$. 
This contradicts the maximality of $U$. 
Thus, this case is impossible and we are done.
\end{proof}

\subsection{Proof of \cref{th:stv}}\label{proof-of-th:stv}
Recall that in our current notation $T_3(1,0)$ stands for the two-edge path $P = (s,t,v)$ with roots $s$, $t$.
 \cref{th:stv} states that if a graph $G$ contains an induced $P_3$ then there exists an avoidable copy of $T_3(1,0)$ in $G$, provided that $G$ is either $C_5$-free or subcubic.

In the proof we use the following definitions.

\begin{defn}\label{def:avoiding-bags}
A~sequence $\B = (B_1, B_2, \dots, B_m)$ such that $\emptyset\neq B_i\subseteq V(G)$ for all $i\in \{1,\ldots, m\}$ is a  \emph{sequence of avoiding bags} in a graph $G$ if
\begin{enumerate}[1${}^{\circ}$] 
\item a subgraph of $G$ induced by $B_i$ is connected for any $i>0$;
\item  $B_i\cap N_G[B_j] = \es$ for any $i\ne j$;
\item\label{item-core-N2} the \emph{core} $C$ of $\B$, that is, the subgraph of $G$ induced by $V(G)\setminus \bigcup_j N_G[B_j]$, contains an induced $P_3$.
\end{enumerate}
The \emph{rank} of $\B = (B_1, B_2, \dots, B_m)$ is the integer sequence $\rank \B= (|B_1|, |B_2|, \dots, |B_m|)$. 
\end{defn}

Suppose that the theorem is false and let $G$ be a minimal counterexample.
That is, $G$ is $C_5$-free or subcubic and $G$ contains a copy of $T_3(1,0)$ but no avoidable copy, and this does not happen for any graph with fewer vertices than $G$. 
Note that each subgraph of $G$ is also $C_5$-free or subcubic.

Since $G$ contains no avoidable copy of $T_3(1,0)$, no copy of $T_3(1,0)$ in $G$ is simplicial. 
Let $(P,s,t)$ be a copy of $T_3(1,0)$ in $G$. 
Since the copy is not simplicial, it has an extension $(P',s',t')$.
Note that $(V(P)\cup \{t'\})\setminus N[s']$ induces a $P_3$ in $G-N[s']$, and hence  $(\{s'\})$ is a non-empty sequence of avoiding bags.

For the rest of the proof we fix a non-empty sequence of avoiding bags $\B = ( B_1, B_2, \dots, B_m)$ which has maximal rank w.r.t.~the lexicographical order on integer sequences. 
Furthermore, let $C$ be the core of $\B$ as defined in~\cref{def:avoiding-bags}.

The following technical lemmas hold for every  graph $G$ that is a minimal counterexample to \cref{conj-stv}. 
We do not use in the proofs that $G$ is $C_5$-free or subcubic.

\begin{lem}\label{prop:same-bag}
Let $x,y\in V(C)$, $x',y'\notin V(C)$, $xx', yy'\in E(G)$, $xy', yx', x'y'\notin E(G)$. 
Then $x',y'\in N(B_j)$ for some $j\in \{1,\ldots, m\}$.
\end{lem}

\begin{proof}
Note that since $x'\not\in V(C)$, there exists some $r\in \{1,\ldots, m\}$ such that $x'\in N(B_r)$, and similarly for $y'$.
Let $i = \min\{r: x'\in N(B_r)\}$ and $j = \min\{r: y'\in N(B_r)\}$. 
Without loss of generality we assume that $i\leq j$.

If $x'\in N(B_j)$, then we are done.
So we may assume that $x'\not\in N(B_j)$.
Furthermore, $y'\in N(B_j)\setminus N(B_i)$ and, for some $y''\in B_j$,  $y''y'y$ is an induced $P_3$ in $G-N[B_i\cup \{x'\}]$. 
Thus,  $(B_1, \dots, B_{i-1}, B_i\cup\{x'\})$ is a sequence of avoiding bags, since $x'\notin N[B_r]$ for $r<i$.
Its rank is greater than $\B$, and we come to a~contradiction with rank-maximality of $\B$.
\end{proof}

\begin{lem}\label{prop:core}
  For any $x\in V(C)$, every connected component of the graph $C - N_C[x] $ is complete.
\end{lem}

\begin{proof}
  A graph does not contain an induced $P_3$ if and only if it is a  disjoint union of complete graphs. 
  If $C - N_C[x] $ contains an induced $P_3$, then $(B_1,\dots, B_m, \{x\})$ is a sequence of avoiding bags having the rank greater than $\B$, a contradiction  with rank-maximality of $\B$.
\end{proof}

\begin{lem}\label{prop:tail-extension}
  Let $xyz$ be an induced $P_3$ in $C$ and $(P',x',y')$ be an extension of the copy $(xyz, x,y)$ of $T_3(1,0)$ in $G$ such that $x'\notin V(C)$. Then $y'\notin V(C)$, too.
\end{lem}
\begin{proof}
  Suppose, for the sake of contradiction, that $y'\in V(C)$. 
  Let $i = \min\{r: x'\in N(B_r)\}$. 
  Note that $y'yz$ is an induced $P_3$ in $C$ and, moreover, $\{y',y,z\}\cap N_G(x')=\es$. Thus, the sequence $(B_1, \dots, B_{i-1}, B_i\cup\{x'\})$ is a sequence of avoiding bags in $G$ having the rank greater than $\B$, a~contradiction.
\end{proof}

\begin{lem}\label{prop:P4-extension}
   Let $xyz$ be an induced $P_3$ in $C$. 
   Then there exists a vertex $z'$ such that $xyzz'$ is an induced $P_4$ in $C$.
\end{lem}

\begin{proof}
  The copy $(zyx, z,y)$ of $T_3(1,0)$ is not avoidable in $G$. 
  Let $(P',z',y')$ be a non-closable extension of it. 
  Then $xyzz'$ is an induced $P_4$ in $G$. 
  To complete the proof we show that $z'\in V(C)$.

  Suppose, for the sake of contradiction, that $z'\notin V(C)$. 
  Then, due to \cref{prop:tail-extension}, $y'\notin V(C)$.  \cref{prop:same-bag} can be applied to $z, y\in V(C)$, $z',y'\notin V(C)$.
  Thus, $z',y'\in N(B_j)$ for some $j\in \{1,\ldots, m\}$, and the extension can be closed by a path in $B_j$, a contradiction.
\end{proof}

\begin{lem}\label{prop:C5-extension}
  Any induced $P_3$ in $C$ is a part of an induced $C_5$ in $C$.
\end{lem}
\begin{proof}
  Let $abc$ be an induced $P_3$ in $C$. 
  Due to \cref{prop:P4-extension}, there exist two induced $P_4$s in $C$ of the form $a^*abc$, $abcc^*$.

  Suppose that for all induced $a^*abc$, $abcc^*$ we have $a^*=c^*$.
  This implies that $a'\notin V(C)$ for each extension $(P', a', b')$ of the copy $(abc, a,b)$ of $T_3(1,0)$ in $G$. Due to \cref{prop:tail-extension}, $b'\notin V(C)$ for each extension. Again, applying  \cref{prop:same-bag} we see that each extension is closable and the copy  $(abc, a,b)$ of $T_3(1,0)$ is avoidable in $G$, a contradiction with the assumption that $G$ has no avoidable copies of $T_3(1,0)$.

  If $a^*\ne c^*$ then $a^*c^*\in E(G)$ due to \cref{prop:core}, since otherwise there exists an induced $P_3$ outside $N_C[a^*]$. Therefore, $\{a^*,a,b,c,c^*\}$ induces a $C_5$ in $C$, and we are done.
\end{proof}

\begin{proof}[Proof of \cref{th:stv}]
\cref{prop:C5-extension} implies that the theorem holds under condition \eqref{th:stv-C5-free}.
Indeed, if $G$ is a $C_5$-free counterexample to the theorem, we come to a contradiction with \cref{prop:C5-extension}.

In fact, since there exists a vertex $v\in B_1$ and $V(C)\subseteq V(G)\setminus N[v]$, a contradiction would also be obtained under a weaker assumption that $G$ is $(C_5+K_1)$-free, where $C_5+K_1$ is the graph obtained from $C_5$ by adding to it an isolated vertex.
Thus, if a graph $G$ does not contain induced $C_5+K_1$ and contains an induced $P_3$ then there exists an avoidable copy of $T_3(1,0)$ in $G$.

We now focus on the proof of the theorem under condition \eqref{th:stv-cubic}. 
Let $G$ be a minimal counterexample. 
Then $G$ is connected. 
Furthermore, $G$ does not contain pendant 
  vertices, since any such vertex is an endpoint of $P_3$ and the corresponding copy of $T_3(1,0)$ is simplicial.
  Observe next that $G$ does not contain triangles (that is, cliques of size three). 
  For the sake of contradiction, assume that $v_0$, $v_1$, $v_2$ are vertices of a triangle. 
  Without loss of generality, since $G$ is connected, there exists an edge $v_0v_3$, $v_3\ne v_i$, $i\in\{0,1,2\}$. 
  The degree of $v_0$ is $3 $, so there are no other edges incident to $v_0$. 
  If $v_3v_i\notin E(G)$, $i\in\{1,2\}$, then the copy $(v_3v_0v_i, v_3,v_0)$ of  $T_3(1,0)$ is simplicial. Therefore $v_0,v_1,v_2,v_3$ form a complete subgraph that is a connected component of $G$, a contradiction with the fact that $G$ is connected and contains an induced $P_3$.
  The absence of triangles implies that $G$ does not contain vertices of degree $2$: any such vertex would be the middle point of a simplicial copy of $T_3(1,0)$.
 We conclude that $G$ is cubic.

Recall that $\B = (B_1,\dots,B_m)$ is a sequence of avoiding bags with maximal rank, and $C$ is its core.
Take an induced $P_3$ in $C$ (such a $P_3$ exists due to \cref{def:avoiding-bags}) and an induced $C_5$ in $C$ extending $P_3$ (such a $C_5$ exists due to \cref{prop:C5-extension}). 
Let $u_0,u_1, u_2,u_3, u_4$ be the vertices of this $C_5$ numbered consequently along the cycle. 
Since $G$ is cubic and the cycle is induced, the vertex $u_i$, where $0\le i\le 4$, has a unique neighbor $w_i$ outside the cycle.

Next, we prove that $w_i\notin V(C)$ for all $0\le i\le 4$. 
Assume, for the sake of contradiction, that $w_0\in V(C)$. 
Since $G$ does not contain triangles, $u_4w_0\notin E(G)$. 
Then $w_0u_2\in E(G)$, since otherwise $w_0$ and $u_4$ would form a pair of non-adjacent vertices in the same component of $C-N_C[u_2]$, contradicting \cref{prop:core}.
Since there are no triangles in $G$, vertex $u_3$ is not adjacent to any vertex in the set $\{w_0,u_0,u_1\}$.
We come to a contradiction with \cref{prop:core} since $w_0$ and $u_1$ are not neighbors.

Let $P = (u_0,u_1,u_2)$.
Since the copy $(P, u_0, u_1)$ of $T_3(1,0)$ is not avoidable in $G$, it has a non-closable extension $(P', a', b')$. 
Since $u_1$ has degree three in $G$ and 
$b'\not\in \{u_0,u_2\}$, 
we infer that $b' = w_1$.
We show next that $a' = u_4$.
Suppose that this is not the case.
Since $u_0$ has degree three in $G$ and $a'\not\in \{u_1,u_4\}$, we must have $a' = w_0$.
However, applying \cref{prop:same-bag} to $(x,y,x',y') = (u_0,u_1,a',b')$ implies that $a',b'\in N(B_j)$ for some $j\in \{1,\ldots, m\}$, contradicting the assumption that $(P', a', b')$ is not closable.
This shows that $a' = u_4$, as claimed.
Thus, $w_1\ne w_4$, since otherwise $a'$ would be adjacent to $b' = w_1$.
Finally, observe that the extension $(P', a', b')$ is closable.
If $w_1w_4\in E(G)$, then the extension can be closed using the path $(a'=u_4,w_4,w_1=b')$, since $w_0\ne w_4$, which holds since $G$ has no triangles. 
Otherwise we apply \cref{prop:same-bag}. 
We come to a contradiction that completes the proof.
\end{proof}

\subsection{More candidates for inherence}\label{ssec:more}

\bigskip
Let us recall that we are not aware of any non-inherent two-rooted combs of type III. 
Additionally, the inherence of the following small two-rooted graphs is also open.
\begin{itemize}
    \item Three two-rooted graphs of order $6$, all of which are two-rooted combs of type I (the first two are paths): $T_1(2,0,4)$, $T_1(4,0,2)$, $T_1(1,1,3)$ (see \cref{picorder6}).
    \begin{figure}[!h]
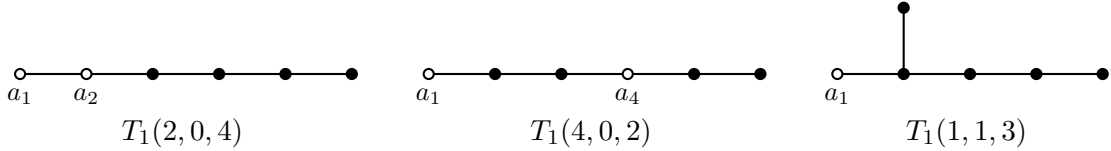

  \begin{center}
    \begin{tabular}{c@{\qquad}c@{\qquad}c}
      \mpfile{pi2rg}{601} & \mpfile{pi2rg}{602}& \mpfile{pi2rg}{603}\\
       $T_1(2,0,4)$& $T_1(4,0,2)$& $T_1(1,1,3)$
    \end{tabular}
  \end{center}
  \caption{Open cases of order $6$ (except combs of type III)}\label{picorder6}
\end{figure}

    \item The following $7$-vertex two-rooted graphs (see~\cref{picorder7}): 
    \begin{itemize}
        \item paths (that is, two-rooted combs of type I with no teeth): 
        $T_1(5,0,2)$, 
        $T_1(3,0,4)$, 
        $T_1(2,0,5)$,
        \item two-rooted combs of type I with one tooth:
        $T_1(4,1,1)$, 
        $T_1(3,1,2)$, 
        $T_1(2,1,3)$, 
        $T_1(1,1,4)$, 
        \item a two-rooted comb of type I with two teeth, $T_1(1,2,2)$,
        \item the extended claw, and the rake $T_4(1)$.
    \end{itemize}
\end{itemize}
\begin{figure}[!h]
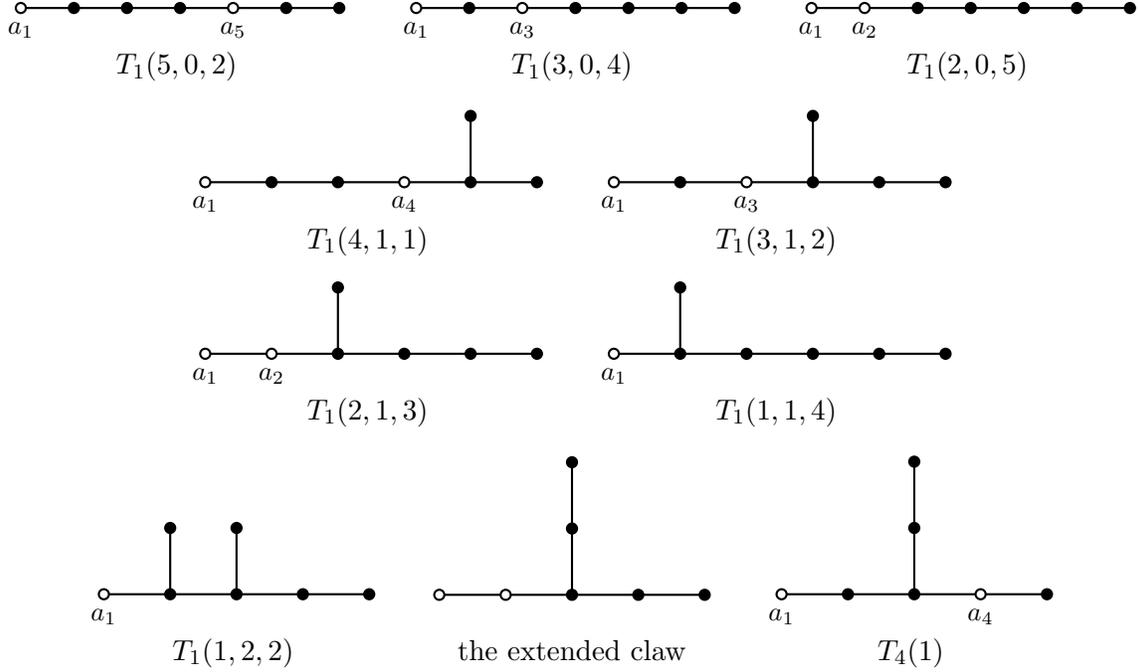

  \begin{center}
    \begin{tabular}{c@{\qquad}c@{\qquad}c}
      \mpfile{pi2rg}{701} & \mpfile{pi2rg}{702}& \mpfile{pi2rg}{703}\\
      $T_1(5,0,2)$&        $T_1(3,0,4)$&        $T_1(2,0,5)$
    \end{tabular}
    \vskip3mm
    \begin{tabular}{c@{\qquad}c}
      \mpfile{pi2rg}{711} & \mpfile{pi2rg}{712}\\
 $T_1(4,1,1)$&       $T_1(3,1,2)$\\[3mm]
         \mpfile{pi2rg}{713} & \mpfile{pi2rg}{714}\\
        $T_1(2,1,3)$&         $T_1(1,1,4)$
   \end{tabular}
    \vskip3mm
    \begin{tabular}{c@{\qquad}c}
      \mpfile{pi2rg}{721}& \raisebox{8.8pt}{\mpfile{pi2rg}{741}}\\
      $T_1(1,2,2)$ &  the extended claw
   \end{tabular}
  \end{center}
  \caption{Open cases of order $7$ (except  combs of type III)}\label{picorder7}
\end{figure}

\subsection{Inherent disconnected two-rooted graphs}\label{ssec:disconnected}

Up to now, we have considered only connected two-rooted graphs. 
However, inherent graphs may be disconnected.

\begin{prop}\label{prop:disconnected}
  Let $(H,s,t)$ be an inherent two-rooted graph. 
  Then for any graph $H'$, the two-rooted graph $(H+ H',s,t)$ is also inherent.
\end{prop}

\begin{proof}
  Suppose $(\hat H,\hat s,\hat t)$ is inherent, and let $\hat{H'}$ be an arbitrary graph. 
  Furthermore, let $G$ be a graph and let  $( H+ H', s, t)$ be a copy of $(\hat H+\hat H',\hat s,\hat t)$ in $G$.
  Now set $G'=G-N[V(H')]$, and observe that $G'$ admits a copy of $(\hat H,\hat s,\hat t)$. 
  Recall that $(\hat H,\hat s,\hat t)$ is inherent, so let $(H^*,s^*,t^*)$ be an avoidable copy of  $(\hat H,\hat s,\hat t)$ in $G'$.
  Finally, observe that $(H^*\cup H',s^*,t^*)$ is avoidable in~$G$.
\end{proof}

\Cref{prop:disconnected} provides many new examples of inherent graphs.
Also, there might be more inherent disconnected two-rooted graphs; see \cref{sec:open-questions}.

\section{Open questions and possible generalizations}\label{sec:open-questions}

We have concentrated on the case of connected two-rooted graphs.
This is justified by the following two conjectures related to the general case.

\begin{conjecture}
A two-rooted graph $(\hat H,\hat s,\hat t)$ is not inherent if $\hat s$ and $\hat t$ are not in the same component of $H$.
\end{conjecture}

\begin{conjecture}
If $(\hat H+ \hat H',\hat s,\hat t)$ is  inherent and $\hat s,\hat t$ are in $\hat H$, then $(\hat H,\hat s,\hat t)$ is also inherent.
\end{conjecture}

Recall that \cref{conj-stv} states that the two-rooted graph $T_3(1,0)$ is inherent.
More generally, we ask the following.

\begin{question}
Which two-rooted combs of type III are inherent?
\end{question}

The following much more general questions are also still open.

\begin{question}
Is the problem of recognizing if a given two-rooted graph is inherent (resp.~PE-inherent) decidable? 
If so, is it solvable in polynomial time?
\end{question}

In conclusion, we mention some possible generalizations of the considered concepts from two-rooted graphs to \emph{$k$-rooted graphs}, that is, graphs $H$ with ordered $k$-tuples of vertices, for any integer $k\ge 2$.

Given a graph $G$, a $k$-rooted graph $(\hat H, \hat s_1,\ldots, \hat s_k)$, and a copy $(H,s_1,\ldots, s_k)$ of it in $G$, an \emph{extension of $(H,s_1,\ldots, s_k)$ in $G$} is any $k$-rooted graph $(H',s_1',\ldots, s_k')$ such that $H'$ is a subgraph of $G$ obtained from $H$ by adding to it $k$  pendant edges $s_1s_1',\ldots, s_ks_k'$.
More precisely, $V(H') = V(H)\cup\{s_1',\ldots, s_k'\}$, vertices $s_i'$ and $s_j'$ are distinct for all $i\neq j$, the graph obtained from $H'$ by deleting $\{s_1',\ldots, s_k'\}$ is $H$, and $s_i$ is the unique neighbor of $s_i'$ in $H'$ for all $i\in \{1,\ldots, k\}$.
Furthermore, we say that an extension $(H',s_1',\ldots, s_k')$ of a copy $(H,s_1,\ldots, s_k)$ of a $k$-rooted graph in a graph $G$ is \emph{closable} if there exists a component $C$ of the graph $G-N[V(H)]$ such that each $s_i'$ has a neighbor in $V(C)$.\footnote{There are several ways to generalize the concepts of extension and/or closability from two-rooted graphs to $k$-rooted graphs;
however, the definitions given above seem to be the most natural ones.}
Finally, having defined the concepts of extensions and closability in the context of $k$-rooted graphs,
the concepts of avoidability and inherence can be defined in the same way as in the case of two-rooted graphs.

The method of pendant extensions and the notion of PE-inherence can also be generalized.
Fix $k> 0$ and a $k$-rooted graph $(H, s_1, \ldots, s_k)$ (some roots may coincide).
Let us call it \emph{amoeba}. 
Amoebas replicate as follows. 
Initially, add a pendant edge to each root. 
In general, we add new pendant edges to the current graph to ensure that every replica (copy) of the initial amoeba admits an extension.
An amoeba is said to be \emph{confined} (with respect to the replication process) if the replication process is finite, and \emph{PE-inherent}, otherwise. 

\begin{rproblem}
Characterize PE-inherent amoebas. 
\end{rproblem}

Several results from this paper extend to the setting of amoebas (for arbitrary $k$). 
In particular, all connected PE-inherent amoebas are trees with maximum degree at most $k+1$.
The one-rooted case, $k=1$, is very simple:
all connected one-rooted amoebas are confined by the replication process, except paths with the root at an end. 

Every PE-inherent ($k$-rooted) amoeba is also PE-inherent for bigger values of $k$, as long as we leave the ``initial'' roots in place. 
Related to this, it may be interesting to study \emph{minimally PE-inherent} amoebas, that is, amoebas that are PE-inherent but are not PE-inherent with respect to any proper nonempty subset of the roots. 

We refer to~\cite{amoeba2024} for some initial results about amoebas.
\newpage
\paragraph{Acknowledgments.}

The authors are grateful to the anonymous reviewers for helpful suggestions.

This work is supported in part by the Slovenian Research and Innovation Agency (I0-0035, research programs P1-0285 and P1-0383 and research projects J1-3001, J1-3002, J1-3003, J1-4008, J1-4084, J1-60012, J5-4596, and N1-0370) and by the research program CogniCom (0013103) at the University of Primorska.
The work was initiated in the framework of a bilateral project between Slovenia and Russia, financed by the Slovenian Research and Innovation Agency (BI-RU/19-20-022).
Part of the work for this paper was done in the framework of bilateral projects between Slovenia and the USA, financed by the Slovenian Research and Innovation Agency (BI-US/22--24--093, BI-US/22--24--149, and BI-US/24--26--088).
The first and fourth authors were working within the framework of the HSE University Basic Research Program. 
The work of the fourth author was also supported in part by the state assignment topic no.~FFNG-2024-0003.

\end{document}